\documentclass[12pt,a4paper,reqno]{amsart} 

\usepackage{mathrsfs}
\usepackage{amsmath}
\usepackage{latexsym}
\usepackage{amssymb}
\usepackage{mathrsfs}
\usepackage{cite}
\usepackage{calc}                   

\newcommand{\cdo}{\, \cdot \, }

\newcommand{\eabs}[1]{\langle #1\rangle}
\newcommand{\scal}[2]{\langle #1,#2\rangle}
\newcommand{\Scal}[2]{\left \langle #1,#2\right \rangle}
\newcommand{\rr}[1]{\mathbf R^{#1}}
\newcommand{\zz}[1]{\mathbf Z^{#1}}
\newcommand{\nn}[1]{\mathbf N^{#1}}

\newcommand{\opBJ}{\operatorname*{Op}\nolimits_{\mathrm{BJ}}}
\newcommand{\opp}[1]{\operatorname*{Op}\nolimits_{#1}}
\newcommand{\op}{\operatorname*{Op}\nolimits}
\newcommand{\nm}[2]{\Vert #1\Vert _{#2}}

\newcommand{\nmm}[1]{\Vert #1\Vert}

\newcommand{\BJ}{\operatorname{BJ}}
\newcommand{\sinc}{\operatorname{sinc}}
\newcommand{\ep}{\varepsilon}
\newcommand{\SG}{\operatorname{SG}}

\newcommand{\vrum}{\vspace{0.1cm}}
\newcommand{\abp}[1]{\vert #1\vert}

\newcommand{\sets}[2]{\{ \, #1\, ;\, #2\, \} }

\newcommand{\maclB}{\mathcal B}

\newcommand{\maclK}{\mathcal K}
\newcommand{\maclS}{\mathcal S}
\newcommand{\mascB}{\mathscr B}

\newcommand{\mascF}{\mathscr F}
\newcommand{\mascI}{\mathscr I}
\newcommand{\mascN}{\mathscr N}
\newcommand{\mascP}{\mathscr P}
\newcommand{\mascS}{\mathscr S}

\newcommand{\mabfj}{\boldsymbol j}
\newcommand{\mabfk}{\boldsymbol k}
\newcommand{\mabfm}{\boldsymbol m}
\newcommand{\mabfn}{\boldsymbol n}

\providecommand{\U}[1]{\protect\rule{.1in}{.1in}}

\numberwithin{equation}{section}          

\newtheorem{thm}{Theorem}
\numberwithin{thm}{section}

\newcommand{\rubrik}{}
\newtheorem{prop}[thm]{Proposition}

\newtheorem{lemma}[thm]{Lemma}

\theoremstyle{definition}

\newtheorem{defn}[thm]{Definition}

\theoremstyle{remark}
\newtheorem{rem}[thm]{Remark}

\author{Maurice de Gosson}

\address{Faculty of Mathematics, NuHAG,
University of Vienna, Vienna, Austria}

\email{maurice.degosson@gmail.com}

\author{Joachim Toft}

\address{Department of Mathematics,
Linn{\ae}us University, V{\"a}xj{\"o}, Sweden}

\email{joachim.toft@lnu.se}

\title{Continuity properties for Born-Jordan operators with symbols in
H{\"o}rmander classes and modulation spaces}

\begin{document}

\begin{abstract}
We show that the Weyl symbol of a Born-Jordan operator
is in the same class as the Born-Jordan symbol, when
H{\"o}rmander symbols and
certain types of modulation spaces are used as symbol
classes. We use these properties to carry over
continuity, nuclearity and Schatten-von Neumann properties to the
Born-Jordan calculus.
\end{abstract}



\maketitle

\section{Introduction}\label{sec0}


\par

A fundamental question in quantum mechanics concerns quantization. That is,
finding rules which takes observables $a$ in classical mechanics into
corresponding observables $T_a$ in quantum mechanics.
Usually $a$ is a
function of the location $x\in \rr d$ (the configuration variable) and the momentum
$\xi \in \rr d$, and $T_a$ is a linear operator which acts between suitable Hilbert
spaces of functions on $\rr d$.

\par

The Born-Jordan quantization
\begin{equation}\label{Eq:ClassicBJ}
\op _{\BJ}(a) = \frac 1{l+1}\sum _{k=0}^l D _j^{l-k}\circ x_j^m\circ D _j^k,\quad
a(x,\xi ) = x_j^m\xi _j^l
\end{equation}
introduced in early days by Born and Jordan in \cite{BJ},
is nowadays considered as an important quantization
rule (see e.{\,}g.
\cite{Transam,deGosson2,golu}). In a context of the calculus of
pseudo-differential
operators, Born-Jordan quantization is given by
\begin{equation}\label{Eq:PseudoBJ1}
\op _{\BJ}(a) \equiv \int _0^1\op _t (a)\, dt,
\end{equation}
where the \emph{symbol} $a(x,\xi )$ is allowed to belong to more
general classes compared to \eqref{Eq:ClassicBJ} (see Section
\ref{sec1} for details and notations).
Here $\op _t(a)$ is the ($t$-Shubin) pseudo-differential operator with
symbol $a$,
given by
\begin{align*}
(\op _t (a)f)(x) &= (2\pi )^{-d}\iint _{\rr {2d}} a((1-t )x+t y,\xi )f(y)
e^{i\scal {x-y}\xi}\, dyd\xi .
\intertext{By considering the right-hand side of \eqref{Eq:PseudoBJ1}
as a mean-value
of the $\op _t(a)$ operators, the median of the integrand is the Weyl operator
(Weyl quantization) $\op ^w(a)=\op _{1/2}(a)$,
given by}
(\op ^w (a)f)(x) &= (2\pi )^{-d}\iint _{\rr {2d}} a(\textstyle {\frac 12}(x+y),\xi )f(y)
e^{i\scal {x-y}\xi}\, dyd\xi ,
\end{align*}
which is also considered as a suitable quantization rule. For example,
among the operator representations above, only
$\op _{\BJ}(a)$ and $\op ^w(a)$ possess the property to necessarily
being self-adjoint (Hermit operators) when $a$ is real-valued.

\par

By using
\begin{equation}\label{Eq:CalculTransfer}
\op _t(a)=\op ^w(a_t),
\quad \text{where}\quad
a_t=e^{i(\frac 12-t)\scal {D_\xi}{D_x}}a,
\end{equation}
the formula \eqref{Eq:PseudoBJ1} becomes
\begin{equation}\label{Eq:WeylBJ1}
\begin{aligned}
\op _{\BJ}(a) &= \op ^w(a_{\BJ}),
\\[1ex]
a_{\BJ}
&=
\int _0^1 e^{i(\frac 12-t)\scal {D_\xi}{D_x}}a\, dt
= 
2\sinc (\textstyle{\frac 12}\scal {D_\xi}{D_x})a,
\end{aligned}
\end{equation}
which puts the Born-Jordan quantization within the frame of the
Weyl calculus of pseudo-differential operators. Here $\sinc t$ is the
sinc function, given by
$$
\sinc t =
\begin{cases}
\frac {\sin t}t, &t\neq 0,
\\[0.5ex]
1, &t=0.
\end{cases}
$$

\par

During the last 10 years, Born-Jordan quantization is also recognized in
time-frequency analysis. In this field, time-frequency resolutions by
Wigner distributions,
i.{\,}e. simultaneously localizations of the time and frequency for
signals, are  essential.
A problem here concerns interpolating frequencies or so-called
ghost frequencies,
which originate from interference of existing frequencies but
are absent in the signal, but are present in the graphs of
their resolutions. (See \cite{bogetal1,bogetal2,Tur}.)
Especially we remark that in \cite{Tur}, Turunen shows that the
time-frequency resolutions
usually becomes significantly more clear when using Born-Jordan
versions  $W_{\BJ}(f,g)$
in place of classical time-frequency resolutions like the Wigner
distributions $W(f,g)$.
For example, it is shown in \cite{Tur} that the ghost frequencies
miraculously almost disappear when using suitable resolutions
based on the  $W_{\BJ}$ transform. See also \cite{CoGoNi4}
for other related facts.

\medspace

The impact of Born-Jordan operators in quantization and
time-frequency analysis leads to
questions on continuity for such operators. In quantization,
it is suitable to consider
general H{\"o}rmander classes $S(m,g)$ on the phase space.
Recall that $S(m,g)$ agrees with classical symbol classes like
$S^r_{\rho ,\delta }(\rr {2d})$, $\SG$-classes or Shubin classes,
by choosing the Riemannian metric $g$ and the weight function
$m$ in appropriate ways.
In time-frequency analysis, it is suitable to use (classical)
modulation spaces, $M^{p,q}_{(\omega )}(\rr {2d})$ as symbol
classes, because they are especially adapted for energy
estimates for time-frequency representations (cf. e.{\,}g.
\cite{Str}). These spaces were introduced in \cite{Fe4} by
Feichtinger and are obtained by imposing an
$L^{p,q}_{(\omega )}$ condition on the short-time Fourier
transform on the involved functions and (ultra-)distributions.
(See also \cite{FG1,GaSa,Gc2} and the references therein
for more facts on modulation spaces.)

\par

In Sections \ref{sec2}--\ref{sec5} we deduce several types of
continuity properties for Born-Jordan operators. In similar way
as in e.{\,}g. \cite{CoGoNi1,CoGoNi2,CoGoNi3}, the main idea
is to use \eqref{Eq:PseudoBJ1}
to carry over continuity properties in pseudo-differential calculus to
Born-Jordan operators.
In Section \ref{sec2} we consider Born-Jordan operators with symbols in
the Schwartz space or in certain Gelfand-Shilov spaces and their duals.
For example, we regain the fact from \cite{CoGoNi1} that 
\eqref{Eq:PseudoBJ1} leads to
\begin{equation}\label{Eq:BJSymbDistRel}
a_{\BJ}\in \mascS '(\rr {2d}) \quad
\quad \text{when}\quad
a\in \mascS '(\rr {2d}),
\end{equation}
where $\mascS '(\rr {2d})$ is the set of tempered distributions
on $\rr {2d}$.
(See Theorems \ref{Thm:BJopsDistrMap} and \ref{BJGSSymbols}.)
In Section \ref{sec2} it is proved that
the same holds true with $\maclS _s(\rr {2d})$,
$\Sigma _s(\rr {2d})$ or their duals in place of $\mascS '(\rr {2d})$, where
$\maclS _s(\rr d)$ ($\Sigma _s(\rr d)$) is the Fourier invariant
Gelfand-Shilov space of Roumieu (Beurling) type of order $s>0$ on $\rr d$.
In particular it follows that the following holds true:
\begin{alignat*}{4}
a &\in \mascS '(\rr {2d})&
\quad &\Rightarrow &\quad
\op _{\BJ}(a) \, &:Ê\, &
\mascS (\rr d) &\to \mascS '(\rr d)
\\[1ex]
a &\in \maclS _s'(\rr {2d})&
\quad &\Rightarrow &\quad
\op _{\BJ}(a) \, &:Ê\, &
\maclS _s(\rr d) &\to \maclS _s'(\rr d),
\intertext{and}
a &\in \Sigma _s'(\rr {2d})&
\quad &\Rightarrow &\quad
\op _{\BJ}(a) \, &:Ê\, &
\Sigma _s(\rr d) &\to \Sigma _s'(\rr d)
\end{alignat*}
(whith continuous mappings), because the same continuity
properties hold true with $\op _t(a)$ in place of $\op _{\BJ}(a)$ for every $t$.
We remark that our investigations include more general
Gelfand-Shilov spaces and their distributions, which do not need
to be Fourier invariant.
(See Theorem \ref{Thm:BJopsDistrMap}.)
These properties give a solid basement of our investigations.

\par

In Section \ref{sec3} we consider Born-Jordan operators with symbols in
modulation spaces, and prove that
\begin{equation}\label{Eq:BJSymbDistRel}
a_{\BJ}\in M^{p,q}_{(\omega )}(\rr {2d}) \quad
\quad \text{when}\quad
a\in M^{p,q}_{(\omega )}(\rr {2d})
\end{equation}
when $p,q\in (0,\infty ]$, provided the weight $\omega (x,\xi ,\eta ,y)$
is constant with respect to the $x$ and $\xi$ variables. (See Theorem
\ref{BJModSymbols}.) It is well-known
that for such $\omega$, the map $e^{it\scal {D_\xi}{D_x}}$ is continuous on
$M^{p,q}_{(\omega )}(\rr {2d})$ for every $t$. (Cf. \cite[Proposition 1.9]{Toft21}.)
It follows that 
In the special case \eqref{Eq:BJSymbDistRel} $p,q\ge 1$,  is deduced
by a straight-forward combination of \eqref{Eq:PseudoBJ1},
\eqref{Eq:CalculTransfer} and Minkowski's inequality. For such choices of
$p$ and $q$ and if all weights are trivially equal to one,
then these investigations are related to those in
\cite{CoGoNi2}. For example Theorem \ref{Thm:BornJordOpModCont}
in Section \ref{sec3} overlaps with \cite[Theorem 5.1]{CoGoNi2}. 

\par

In order to reach \eqref{Eq:BJSymbDistRel} in the 
general case, the possible lack of local-convexity of involved spaces,
impose
a more comprehensive analysis compared to the restricted case
$p,q\ge 1$. In our approach, the symbol $a$ in \eqref{Eq:BJSymbDistRel}
is expressed in terms of its Gabor expansion, using the fact that
Gabor theory works properly for modulation spaces when
$p$ and $q$ are allowed to be smaller than $1$. (Cf. \cite{GaSa,Toft9}.)
By inserting such expansions in \eqref{Eq:PseudoBJ1},
\eqref{Eq:CalculTransfer} and performing some refined computations,
we finally land on \eqref{Eq:BJSymbDistRel}.

\par

As a consequence of \eqref{Eq:BJSymbDistRel} we get, e.{\,}g.,
\begin{align*}
a&\in M^{\infty ,q}(\rr {2d}),\ q\le 1,\ p_1,q_1\in [q,\infty]
\\[1ex]
&{\phantom k}\qquad \qquad \qquad \Rightarrow \quad
\op _{BJ}(a)\, :\, M^{p_1,q_1}(\rr d)\to M^{p_1,q_1}(\rr d),
\\[1ex]
a&\in M^{p,\min (p,p')}(\rr {2d}),\ p>0
\quad \Rightarrow \quad
\op _{BJ}(a) \in \mascI _p(L^2(\rr d))
\\[1ex]
a&\in M^{p}(\rr {2d}),\ p\le 1
\quad \Rightarrow \quad
\op _{BJ}(a) \in \mascN _p(M^\infty (\rr d), M^p (\rr d)),
\end{align*}
because the same hold true with $\op _t(a)$ in place of $\op _{BJ}(a)$
(cf. \cite{Toft19,Toft21}). Here $p'$ is the conjugate exponent of $p\in (0,\infty ]$,
given by
$$
p'=
\begin{cases}
\infty  & \text{when}\ p\le 1,
\\[1ex]
\frac p{p-1} & \text{when}\ 1<p<\infty ,
\\[1ex]
1 & \text{when}\ p=\infty .
\end{cases}
$$

\par

In Section \ref{sec4} we deduce
continuity properties for Born-Jordan
operators with symbols in the (general) H{\"o}rmander class $S(m,g)$,
where $g$ is a strongly feasible Riemannian metric  and $m$ is
$(\sigma ,g)$-temperated metric on the
phase space $\rr {2d}$. We prove that if in addition $g$ is \emph{split} in the sense
$g_X(y,-\eta )=g_X(y,\eta )$, then
\begin{equation}\label{Eq:BJHormSymbRel}
a_{\BJ}\in S(m,g) \quad
\quad \text{when}\quad
a\in S(m,g)
\end{equation}
In particular, all continuity properties for pseudo-differential operators
in \cite{BuTo,Hormweyl,Toft4,Toft19} with symbols in $S(m,g)$,
carry over to Born-Jordan operators with symbols in the same class.
In particular it follows that
$$
a\in S(m,g)
\quad \Rightarrow
\begin{cases}
\op _{\BJ}(a) \, :\, \mascS (\rr d)\to \mascS (\rr d),
\\[1ex]
\op _{\BJ}(a) \, :\, \mascS '(\rr d)\to \mascS '(\rr d)
\end{cases}
$$
(see Theorem \ref{BJSchwartzcont}), and that
$$
m\in L^p(\rr {2d}), a\in S(m,g)
\quad \Rightarrow \quad
\op _{\BJ}(a)\in \mascI _p(L^2(\rr d))
$$
(see Theorems \ref{HormSchattenBJ} and \ref{HormSchattenBJcor}),
because the same hold true with $\op ^w(a)$ in place of $\op _{\BJ}(a)$
(see \cite[Theorem 18.6.2]{Horm}, \cite[Theorem 2.9]{BuTo},
\cite[Theorem 4.4]{Toft4}and \cite[Theorem 4.1]{Toft19}).

\par

In the last part of Section \ref{sec4} we deduce classical
lower bound estimates for Born-Jordan
operators. In fact, by asymptotic expansions it follows that if
$a,a_{\BJ}\in S(m,g)$ satisfy \eqref{Eq:WeylBJ1}, then
\begin{equation}\label{Eq:DiffWeylBJSymb}
a-a_{\BJ}\in S(h_g^2m,g),
\end{equation}
where $h_g$ is the Planck's function. This leads to that fundamental lower
bound results carry over from the Weyl
case to Born-Jordan case. In fact, Sharp G{\aa}rding's
and Feffermann-Phong's inequalities as well as H{\"o}rmander's improvement
of Melin's inequality, given by Theorems 18.6.7 and 18.6.8 in \cite{Horm} and
Theorem 6.2 in \cite{Horm0}, are some of the most well-known lower bound 
results in
pseudo-differential calculus. It follows from the small difference between
$a$ and $a_{\BJ}$ in view of \eqref{Eq:DiffWeylBJSymb} that these lower bound
results carry over to Born-Jordan operators
(see e.{\,}g. Theorem \ref{GaardingIneq} in Section \ref{sec4}).

\par

In Section \ref{sec5} we consider Born-Jordan operators of so-called
infinite orders. That is, in contrast to the H{\"o}rmander classes, $S(m,g)$,
the involved symbols are allowed to grow faster than polynomials.
On the other hand, it is assumed that the symbols obey
stronger regularity conditions than what is required in the class $S(m,g)$.
We consider operators with symbols in $\Gamma _{s,\sigma ;0}
^{\sigma ,s}(\rr {2d})$ or $\Gamma _{s,\sigma}^{\sigma ,s;0}(\rr {2d})$,
considered in \cite{AbCaTo} (see Definition \ref{Def:ExtGSSymbClasses}
in Section \ref{sec5}). In \cite{AbCaTo} it is deduced that
pseudo-differential operators with symbols in such classes
are continuous on suitable Gelfand-Shilov spaces and their duals.
In Section \ref{sec5} we use \eqref{Eq:WeylBJ1} to
carry over these continuity properties to Born-Jordan operators
with symbols in $\Gamma _{s,\sigma ;0}^{\sigma ,s}(\rr {2d})$
or $\Gamma _{s,\sigma}^{\sigma ,s;0}(\rr {2d})$.

\par

\section*{Acknowledgement} Maurice de Gosson has been
supported by the Austrian research agency FWF
(grant number P27773).

\par

\section{Preliminaries}\label{sec1}

\par

In this section we start by recalling some facts about Gelfand-Shilov spaces of
functions and distributions. Thereafter, we recall the definition of
pseudo-differential operators and Born-Jordan operators. Some basic properties
for Schatten-von Neumann and nuclear operator classes are then discussed
in Subsection \ref{subsec1.4}. We conclude the session by recalling some facts
on modulation spaces.

\par

\subsection{Gelfand-Shilov spaces and their duals}
We start by recalling some facts about Gelfand-Shilov spaces.
Let $0<h,s,\sigma \in \mathbf R$ be fixed. Then
$\maclS _{s;h}^{\sigma}(\rr d)$ is
the Banach space of all $f\in C^\infty (\rr d)$ such that
\begin{equation}\label{gfseminorm}
\nm f{\maclS _{s;h}^{\sigma}}\equiv \sup_{\alpha ,\beta \in
\mathbf N^d} \sup_{x \in \rr d} \frac {|x^\alpha \partial ^\beta
f(x)|}{h^{|\alpha  + \beta |}\alpha !^s\, \beta !^\sigma}<\infty,
\end{equation}
endowed with the norm \eqref{gfseminorm}.

\par

The \emph{Gelfand-Shilov spaces} $\maclS _{s}^{\sigma}(\rr d)$ and
$\Sigma _{s}^{\sigma}(\rr d)$ are defined as the inductive and projective 
limits respectively of $\maclS _{s;h}^{\sigma}(\rr d)$. This implies that
\begin{equation}\label{GSspacecond1}
\maclS _{s}^{\sigma}(\rr d) = \bigcup _{h>0}\maclS _{s;h}^{\sigma}(\rr d)
\quad \text{and}\quad \Sigma _{s}^{\sigma}(\rr d) =\bigcap _{h>0}
\maclS _{s;h}^{\sigma}(\rr d),
\end{equation}
and that the topology for $\maclS _{s}^{\sigma}(\rr d)$ is the strongest
possible one such that the inclusion map from $\maclS _{s;h}^{\sigma}
(\rr d)$ to $\maclS _{s}^{\sigma}(\rr d)$ is continuous, for every choice 
of $h>0$. The space $\Sigma _{s}^{\sigma}(\rr d)$ is a Fr{\'e}chet space
with seminorms $\nm \cdo {\maclS _{s;h}^{\sigma}}$, $h>0$. Moreover,
$\Sigma _{s}^{\sigma}(\rr d)\neq \{ 0\}$, if and only if $s+\sigma \ge 1$
and $(s,\sigma )\neq (\frac 12,\frac 12)$, and
$\maclS _{s}^{\sigma}(\rr d)\neq \{ 0\}$, if and only
if $s+\sigma \ge 1$.

\par

In terms of the exponential type decays, $\maclS _{s}^{\sigma}(\rr d)$
and $\Sigma _{s} ^{\sigma}(\rr d)$ are characterized as $f \in \maclS
_{s}^{\sigma}(\rr d)$ ($f \in \Sigma _{s}^{\sigma}(\rr d)$), if and only if 
$$
|\partial^\alpha f(x)| \lesssim h^{|\alpha|} (\alpha!)^\sigma  e^{-r|x|^{\frac 1s}}
$$
for some $h,r>0$ (respectively for every $h,r>0$). 
Moreover we recall that for $s <1$ the elements of
$\maclS _{s}^{\sigma}(\rr d)$ admit entire 
extensions to $\mathbf{C}^d$ satisfying suitable exponential bounds,
cf. \cite{GS} for details.

\medspace

The \emph{Gelfand-Shilov distribution spaces} $(\maclS _{s}^{\sigma})'(\rr d)$
and $(\Sigma _{s}^{\sigma})'(\rr d)$ are the projective and inductive limits
respectively of $(\maclS _{s;h}^{\sigma})'(\rr d)$.  This implies that
\begin{equation}\tag*{(\ref{GSspacecond1})$'$}
(\maclS _{s}^{\sigma})'(\rr d) = \bigcap _{h>0}
(\maclS _{s;h}^{\sigma})'(\rr d)
\quad
\text{and}\quad (\Sigma _{s}^{\sigma})'(\rr d) =\bigcup _{h>0}
(\maclS _{s;h}^{\sigma})'(\rr d).
\end{equation}
We remark that in \cite{Pil1} it is proved that $(\maclS _{s}^{\sigma})'(\rr d)$
is the dual of $\maclS _{s,\sigma}(\rr d)$, and $(\Sigma _{s}^{\sigma})'(\rr d)$
is the dual of $\Sigma _{s}^{\sigma}(\rr d)$ (also in topological sense).

\par

For every $s,\sigma >0$ we have
\begin{equation}\label{Eq:GSEmbeddings}
\Sigma _s^\sigma (\rr d)
\hookrightarrow
\maclS _s^\sigma (\rr d)
\hookrightarrow
\Sigma _{s+\ep}^{\sigma +\ep}(\rr d)
\hookrightarrow
\mascS (\rr d)
\end{equation}
for every $\ep >0$. If $s+\sigma \ge 1$, then
the last two inclusions in \eqref{Eq:GSEmbeddings} are dense,
and if in addition $(s,\sigma )\neq (\frac 12,\frac 12)$, then the
first inclusion in \eqref{Eq:GSEmbeddings} is dense.

\par

From these properties it follows that $\mascS '(\rr d)\hookrightarrow
(\maclS _s^\sigma)'(\rr d)$ when $s+\sigma \ge 1$, and if in addition
$(s,\sigma )\neq (\frac 12,\frac 12)$, then $(\maclS _s^\sigma)'(\rr d)
\hookrightarrow (\Sigma _s^\sigma)'(\rr d)$.

\par

The Gelfand-Shilov spaces possess several convenient mapping
properties. For example they are nuclear and invariant under
translations, dilations, and to some extent tensor products
and (partial) Fourier transformations, cf. \cite{GS, Mit, Pil2}).

\par

We also need to involve a broader family of
Gelfand-Shilov spaces. More precisely, for $s_j,\sigma _j\in \mathbf R_+$,
$j=1,2$, the Gelfand-Shilov spaces $\maclS
_{s _1,s_2}^{\sigma _1,\sigma _2}(\rr {d_1+d_2})$ and
$\Sigma _{s _1,s_2}^{\sigma _1,\sigma _2}(\rr {d_1+d_2})$ consist of
all functions $F\in C^\infty (\rr {d_1+d_2})$ such that
\begin{equation}\label{GSExtCond}
|x_1^{\alpha _1}x_2^{\alpha _2}\partial _{x_1}^{\beta _1}
\partial _{x_2}^{\beta _2}F(x_1,x_2)| \lesssim
h^{|\alpha _1+\beta _1|+|\alpha _2+\beta _2|}
\alpha _1!^{s_1}\alpha _2!^{s_2}\beta _1!^{\sigma _1}\beta _2!^{\sigma _2}
\end{equation}
for some $h>0$ respective for every $h>0$. The  topologies, and the duals
\begin{alignat*}{3}
&(\maclS _{s _1,s_2}^{\sigma _1,\sigma _2})'(\rr {d_1+d_2}) &
&\quad \text{and} \quad &
&(\Sigma _{s _1,s_2}^{\sigma _1,\sigma _2})'(\rr {d_1+d_2})
\intertext{of}
&\maclS _{s _1,s_2}^{\sigma _1,\sigma _2}(\rr {d_1+d_2}) &
&\quad \text{and} \quad &
&\Sigma _{s _1,s_2}^{\sigma _1,\sigma _2}(\rr {d_1+d_2}),
\end{alignat*}
respectively, and their topologies
are defined in analogous ways as for the spaces $\maclS _s^\sigma (\rr d)$
and $\Sigma _s^\sigma (\rr d)$ above.

\par

From the inequalities $n!k!\le (n+k!)\le 2^{n+k}n!k!$, it follows by
straight-forward computations that
$\maclS _{s,s}^{\sigma ,\sigma} = \maclS _s^\sigma$,
$\Sigma _{s,s}^{\sigma ,\sigma} = \Sigma _s^\sigma$
and similarly for their duals. For convenience we set
$\maclS _s=\maclS _s^s$ and $\Sigma _s=\Sigma _s^s$.

\par

From now on we let $\mathscr F$ be the Fourier transform,
given by
$$
(\mathscr Ff)(\xi ) = \widehat f(\xi ) \equiv (2\pi )^{-d/2}\int _{\rr
{d}} f(x)e^{-i\scal  x\xi }\, dx
$$
when $f\in L^1(\rr d)$. Here $\scal \cdo \cdo$ denotes the
usual scalar product on $\rr d$. The map $\mathscr F$ extends 
uniquely to homeomorphisms from $\mathscr S'(\rr d)$ to
$\mathscr S'(\rr d)$, from $(\maclS _s^\sigma )'(\rr d)$ to
$(\maclS _\sigma ^s)'(\rr d)$ and from
$(\Sigma _s^\sigma )'(\rr d)$ to $(\Sigma _\sigma ^s)'(\rr d)$.
It also follows that $\mathscr F$ restricts
to homeomorphisms from $\mathscr S(\rr d)$ to
$\mathscr S(\rr d)$, from $\maclS _s^\sigma (\rr d)$ to
$\maclS _\sigma ^s(\rr d)$, from
$\Sigma _s^\sigma (\rr d)$ to $\Sigma _\sigma ^s(\rr d)$,
and to a unitary operator on $L^2(\rr d)$.

\par

\begin{rem}\label{Rem:PartFourGS}
In the same way, if $\mascF _kF$ is the partial Fourier transform
of $F(x_1,x_2)$ with respect to $x_j\in \rr {d_j}$, $s_j,\sigma _j>0$,
$j=1,2$ and $d=d_1+d_2$, then
\begin{align*}
\mascF _1 &: \Sigma _{s_1,s_2}^{\sigma _1,\sigma _2}(\rr d)
\to \Sigma _{\sigma _1,s_2}^{s_1,\sigma _2}(\rr d),
&\quad
\mascF _2 &: \Sigma _{s_1,s_2}^{\sigma _1,\sigma _2}(\rr d)
\to \Sigma _{s_1,\sigma _2}^{\sigma _1,s_2}(\rr d),
\\[1ex]
\mascF _1 &: \maclS _{s_1,s_2}^{\sigma _1,\sigma _2}(\rr d)
\to \maclS _{\sigma _1,s_2}^{s_1,\sigma _2}(\rr d),
&\quad
\mascF _2 &: \maclS _{s_1,s_2}^{\sigma _1,\sigma _2}(\rr d)
\to \maclS _{s_1,\sigma _2}^{\sigma _1,s_2}(\rr d),
\\[1ex]
\mascF _j &: \mascS (\rr d)
\to \mascS (\rr d),
&\quad
\mascF _j &: \mascS '(\rr d)
\to \mascS '(\rr d),
\\[1ex]
\mascF _1 &: (\maclS _{s_1,s_2}^{\sigma _1,\sigma _2})'(\rr d)
\to (\maclS _{\sigma _1,s_2}^{s_1,\sigma _2})'(\rr d),
&\quad
\mascF _2 &: (\maclS _{s_1,s_2}^{\sigma _1,\sigma _2})'(\rr d)
\to (\maclS _{s_1,\sigma _2}^{\sigma _1,s_2})'(\rr d)
\intertext{and}
\mascF _1 &: (\Sigma _{s_1,s_2}^{\sigma _1,\sigma _2})'(\rr d)
\to (\Sigma _{\sigma _1,s_2}^{s_1,\sigma _2})'(\rr d),
&\quad
\mascF _2 &: (\Sigma _{s_1,s_2}^{\sigma _1,\sigma _2})'(\rr d)
\to (\Sigma _{s_1,\sigma _2}^{\sigma _1,s_2})'(\rr d)
\end{align*}
are homeomorphisms, $j=1,2$.
\end{rem}

\medspace

Next we recall some mapping properties of Gelfand-Shilov
spaces under short-time Fourier transforms and $t$-Wigner distributions.
Let $\phi \in \mathscr S(\rr d)$ be fixed. For every $t \in \mathbf R$,
$f\in \mathscr S'(\rr d)$, the \emph{short-time Fourier transform} $V_\phi
f$ is the distribution on $\rr {2d}$ defined by the formula
\begin{equation}\label{defstft}
(V_\phi f)(x,\xi ) =\mascF (f\, \overline{\phi (\cdo -x)})(\xi ) =
(f,\phi (\cdo -x)e^{i\scal \cdo \xi}).
\end{equation}
The \emph{$t$-Wigner distribution} is given by
\begin{equation}\label{wignertdef2}
W_{t} (f,\phi )(x,\xi ) \equiv \mascF (f(x+t\cdo
)\overline{\phi (x-(1-t )\cdo )} )(\xi ).
\end{equation}
We observe that if $f,\phi \in \mascS (\rr d)$, then
$V_\phi f$ and $W_{t} (f,\phi )$ are given by
$$
(V_\phi f)(x,\xi ) =(2\pi )^{-\frac d2}\int _{\rr d}
f(y)\overline{\phi (y -x)} e^{-i\scal y\xi}\, dy
$$
and
$$
W_{t} (f,\phi )(x,\xi ) = (2\pi )^{-\frac d2}\int _{\rr d}
f(x+ty
)\overline{\phi (x-(1-t )y )} e^{-i\scal y\xi}\, dy.
$$

\par

The definition of short-time Fourier transforms and Wigner distributions
extend in different
ways, and possess various kinds of continuity properties. In the context
of test function
spaces and distribution spaces we have the following.

\par

\begin{prop}\label{Prop:WignerDistExt}
Let $s,\sigma >0$ be such that $s+\sigma \ge 1$ and let
$T(t ,f,\phi )\equiv V_\phi f$ or $T(t ,f,\phi )
\equiv W_{t} (f,\phi )$ when $f,\phi \in \mascS (\rr d)$. Then the following
is true:
\begin{enumerate}
\item the map $(t ,f,\phi)\mapsto T(t ,f,\phi )$ is continuous from
$\mathbf R\times \mascS (\rr d)\times \mascS (\rr d)$ to
$\mascS (\rr {2d})$ and restricts to a continuous map from
$\mathbf R\times \maclS _s^\sigma (\rr d)\times \maclS _s^\sigma
(\rr d)$ to $\maclS _{s,\sigma}^{\sigma ,s}(\rr {2d})$;


\item the map $(t ,f,\phi)\mapsto T(t ,f,\phi )$ from
$\mathbf R\times \maclS _s^\sigma (\rr d)\times \maclS _s^\sigma
(\rr d)$ to $\maclS _{s,\sigma}^{\sigma ,s}(\rr {2d})$
extends uniquely to a continuous map
from $\mathbf R\times (\maclS _s^\sigma )'(\rr d)\times
(\maclS _s^\sigma )'(\rr d)$ to
$(\maclS _{s,\sigma}^{\sigma ,s})'(\rr {2d})$ and from
$\mathbf R\times \mascS '(\rr d)\times \mascS '(\rr d)$ to
$\mascS '(\rr {2d})$.
\end{enumerate}
The same holds true for $(s,\sigma )\neq (\frac 12,\frac 12)$
when each $\maclS _s^\sigma$ and $\maclS _{s,\sigma}^{\sigma ,s}$
are replaced by $\Sigma _s^\sigma$ and
$\Sigma _{s,\sigma}^{\sigma ,s}$, respectively.
\end{prop}

\par

Proposition \ref{Prop:WignerDistExt} is essentially available in
the literature (see e.{\,}g. \cite{CPRT10,Toft8}). Since in contrast
we have included the parameter $t$ as a variable, we here recall
the arguments for the $t$-Wigner distribution.

\par

\begin{proof}
We only prove (2). The other cases follow by similar arguments
and are left for the reader.

\par

By the definition we have
$$
T(t ,f,\phi) = (\mascF _2\circ U_t \circ S)(f,\phi ),
$$
where
$$
(U_t F)(x,y) = F(x+t y,x-(1-t )y)
\quad \text{and}\quad
S(f,\phi )= f\otimes \overline{\phi}.
$$
Since it is evident that $S$ is continuous from
$(\maclS _s^\sigma )'(\rr d)\times (\maclS _s^\sigma )'(\rr d)$
to $(\maclS _s^\sigma )'(\rr {2d})$,
that
$(t ,F)\mapsto U_t F$ is continuous from
$\mathbf R \times (\maclS _s^\sigma )'(\rr {2d})$ to
$(\maclS _s^\sigma )'(\rr {2d})$,
it follows from Remark \ref{Rem:PartFourGS} that
$T$ is uniquely defined and continuous from
$\mathbf R\times (\maclS _s^\sigma )'(\rr d)
\times (\maclS _s^\sigma )'(\rr d)$ to
$(\maclS _{s,\sigma}^{\sigma ,s})' (\rr {2d})$.
\end{proof}

\par

\begin{rem}\label{Rem:WignerDistExt}
By the previous proof it also follows that the mappings in
Proposition \ref{Prop:WignerDistExt} are in fact locally uniformly bounded. 
\end{rem}

\par

We also notice that if $T$ is the same as in Proposition
\ref{Prop:WignerDistExt},
then the mappings
\begin{alignat*}{2}
T\, &:\, & \mathbf R\times \mascS '(\rr d)\times \mascS (\rr d) &\to
\mascS '(\rr {2d})\bigcap C^\infty (\rr {2d}),
\\[1ex]
T\, &:\, & \mathbf R\times (\maclS _s^\sigma )'(\rr d)
\times \maclS _s^\sigma (\rr d) &\to
(\maclS _{s,\sigma}^{\sigma ,s})'(\rr {2d})\bigcap C^\infty (\rr {2d}),
\intertext{and}
T\, &:\, & \mathbf R\times (\Sigma _s^\sigma )'(\rr d)
\times \Sigma _s^\sigma(\rr d) &\to
(\Sigma _{s,\sigma}^{\sigma ,s})'(\rr {2d})\bigcap C^\infty (\rr {2d})
\end{alignat*}
are continuous (cf. \cite{AbCaTo,CPRT10,Toft8, Toft10}). 

\par

There are several ways to characterize Gelfand-Shilov
spaces and their distribution spaces. For example, they can
easily be characterized by Hermite functions and other
related functions (cf. e.{\,}g. \cite{GrPiRo,JaEi,Pil1,Pil2}). They can
also be characterized by suitable estimates of their Fourier
and Short-time Fourier transforms (cf.
\cite{ChuChuKim,GrZi,Toft8,Toft10}).

\par

\subsection{Pseudo-differential and Born-Jordan operators}

\par

Let $t \in \mathbf{R}$ be fixed. For any $a\in \mathscr {S}
(\mathbf{R}^{2d})$ (the symbol), the pseudo-differential operator
$\operatorname*{Op}_{t}(a)$ is the linear
and continuous operator on $\mathscr S(\rr d)$, defined by
\begin{equation}
\op _{t}(a)f(x)
=
(2\pi)^{-d}\iint a((1-t )x+t y,\xi)f(y)e^{i\scal {x-y} \xi}
\, dyd\xi 
\label{opta1}
\end{equation}
when $f\in \mathscr S(\rr d)$. By straight-forward computations it follows that
\begin{equation}\label{Eq:OpTauEquiv}
(\op _t (a)f,g)_{L^2(\rr d)} = (2\pi )^{-d/2}(a,W_t (g,f))_{L^2(\rr {2d})}
\end{equation}
when $g\in \mathscr S(\rr d)$.

\par

If more generally $a\in \mathscr S'(\rr {2d})$,
then $\operatorname*{Op}_{t}(a)$ is the linear and continuous operator from
$\mathscr S(\rr d)$ to $\mathscr S'(\rr d)$ such that $(\op _t (a)f,g)$ is equal to
the right-hand side of \eqref{Eq:OpTauEquiv} when $f,g\in \mathscr S(\rr d)$.
This makes sense, in view of the continuity properties for the Wigner distribution,
described above. Similar facts hold true with either $\maclS _{s_1}$ or
$\Sigma _{s_2}$ in place of $\mascS$ at each occurrence, when
$s_1\ge \frac 12$ and $s_2>\frac 12$.

\par

We recall that the Born-Jordan operator $\op _{\BJ}(a)$
with symbol $a$ is given by \eqref{Eq:PseudoBJ1}. It follows from \eqref{Eq:OpTauEquiv} that
\begin{align}
(\op _{\BJ} (a)f,g)_{L^2(\rr d)}
&=
(2\pi )^{-d/2}(a,W_{\BJ} (g,f))_{L^2(\rr {2d})}
\label{Eq:WignerBJIntForm}
\intertext{with}
W_{\BJ} (g,f) &= \int _0^1 W_t (g,f)\, dt.
\label{Eq:WignerPseudoRel}
\end{align}

\par

\subsection{Schatten-von Neumann classes and nuclear operators}\label{subsec1.4}

\par

Before giving the general definition of Schatten-von Neumann classes we
recall some facts on quasi-Banach spaces. A quasi-norm $\| \cdo \|_{\mascB}$ of order $r \in (0,1]$ on the
vector-space $\mascB$ is a nonnegative functional on $\mascB$ which satisfies
\begin{alignat}{2}
 \nm {f+g}{\mascB} &\le 2^{\frac 1r-1}(\nm {f}{\mascB} + \nm {g}{\mascB}), &
\quad f,g &\in \mascB ,
\label{Eq:WeakTriangle1}
\\[1ex]
\nm {\alpha \cdot f}{\mascB} &= |\alpha| \cdot \nm f{\mascB},
& \quad \alpha &\in \mathbf{C},
\quad  f \in \mascB
\notag
\intertext{and}
\nm f{\mascB} &= 0\quad  \Leftrightarrow \quad f=0. & &
\notag
\end{alignat}

\par

The vector space $\mascB$ is called a quasi-Banach space if it is a
complete quasi-normed space. If $\mascB$ is a quasi-Banach space
with quasi-norm satisfying the weak triangle inequality
\eqref{Eq:WeakTriangle1},
then by \cite{Aik,Rol} there is an equivalent quasi-norm to $\nm \cdo {\mascB}$
which additionally satisfies
\begin{align}\label{Eq:WeakTriangle2}
\nm {f+g}{\mascB}^r \le \nm {f}{\mascB}^r + \nm {g}{\mascB}^r, 
\quad f,g \in \mascB .
\end{align}
From now on we always assume that the quasi-norm of the quasi-Banach space $\mascB$
is chosen in such a way that both \eqref{Eq:WeakTriangle1} and \eqref{Eq:WeakTriangle2}
hold.

\par

Let $\mascB _1$ and $\mascB _2$ be (quasi-)Banach spaces and let $T$ be a
linear operator from $\mascB  _1$ to $\mascB _2$. The singular value of
$T$ of order $j\ge 1$ is defined as
$$
\sigma _j(T) = \sigma _j(T;\mascB _1 , \mascB _2 )
\equiv \inf \nm {T-T_0}{\mascB _1\to \mascB _2},
$$
where the infimum is taken over all linear operators $T_0$ from $\mascB _1$ to
$\mascB _2$ of rank at most $j-1$. (Cf. e.{\,}g. \cite{Si,BS,Toft}.)
The operator $T$ is said to be a Schatten-von Neumann operator of order
$p\in (0,\infty ]$ if
\begin{equation}\label{SchattenNormBanach}
\nm T{\mascI _p(\mascB _1,\mascB _2)} \equiv \nm {\{Ê\sigma _j(T)Ê\}Ê_{j\ge 1}}{\ell ^p}
\end{equation}
is finite. The set of Schatten-von Neumann operators from $\mascB _1$ to $\mascB _2$
of order $p\in (0,\infty ]$ is denoted by $\mascI _p(\mascB _1,\mascB_2)$.
We observe that $\mascI _p(\mascB _1,\mascB _2)$ is contained in
$\maclK (\mascB _1,\mascB _2)$, the set of compact operators from $\mascB _1$
to $\mascB _2$, when $p<\infty$. Furthermore, $\mascI _\infty (\mascB _1,\mascB _2)$
agrees with $\maclB (\mascB _1,\mascB _2)$, the set of linear bounded operators
from $\mascB _1$ to $\mascB _2$. For conveniency we set $\mascI _\infty (\mascB ,\mascB )
=\mascI _\infty (\mascB )$.

\par

Next we define nuclear operators. Let $\mascB _0$ be a Banach space with dual
$\mascB _0'$, $\mascB$ be a quasi-Banach space,
$r\in (0,1]$ and let $T$ be a linear and continuous operator from $\mascB _0$
to $\mascB$. Then $T$ is called \emph{$r$-nuclear} from $\mascB _0$ to
$\mascB$, if there are sequences $\{ \ep _j\} _{j=1}^\infty \subseteq \mascB _0'$
and $\{ e_j\} _{j=1}^\infty \subseteq \mascB$ such that
\begin{align}
&T = \sum _{j=1}^\infty e_j\otimes \ep _j\label{Eq:rNuclDef}
\intertext{with convergence in $\maclB (\mascB _0,\mascB )$, and}
&
\sum _{j=1}^\infty \nm {\ep _j}{\mascB _0'}^r\nm {e_j}{\mascB}^r<\infty .
\label{Eq:NuclQNormEst}
\end{align}
Here $T$ in \eqref{Eq:rNuclDef} should be interpreted as the operator
$$
Tf = \sum _{j=1}^\infty \scal f{\ep _j}e_j,\qquad f\in \mascB _0,
$$
which is well-defined when \eqref{Eq:NuclQNormEst} holds.
The set of $r$-nuclear operators from $\mascB _0$ to
$\mascB$ is denoted by $\mascN _r(\mascB _0,\mascB )$, and we equip
this set by the quasi-norm
$$
\nm T{\mascN _r(\mascB _0,\mascB )} \equiv \inf 
\left (
\sum _{j=1}^\infty \nm {\ep _j}{\mascB _0'}^r\nm {e_j}{\mascB}^r
\right ) ^{\frac 1r},
$$
where the infimum is taken over all representatives $\{ \ep _j\} _{j=1}^\infty \subseteq \mascB _0'$
and $\{ e_j\} _{j=1}^\infty \subseteq \mascB$ such that \eqref{Eq:rNuclDef} and
\eqref{Eq:NuclQNormEst} hold true.

\par

Later on we need the following result which shows that $p$-nuclearity is
stable under linear continuous mappings. Here and in what follows we write
$g\lesssim h$ when $g(\theta )\le ch(\theta )$
for some constant $c>0$ which is independent of $\theta$ in
the domains of $g$ and $h$. We also let
$g\asymp h$ when $g\lesssim h$ and $h\lesssim g$.

\par

\begin{prop}\label{Prop:NuclQBanach}
Let $p \in (0,\infty ]$, $r\in (0,1]$, $\mascB _k$ be quasi-Banach spaces of order
$p$, $\mascB _{0,k}$ be Banach spaces, $k=1,2$, and let
$$
T_1\, :\, \mascB _{0,2}\to \mascB _{0,1}
\quadÊ\text{and}\quad
T_2\, :\, \mascB _1\to \mascB _2
$$
be continuous. Then the following is true:
\begin{enumerate}
\item if $T\in \mascI _p(\mascB _{0,1},\mascB _1)$, then
$T_2\circ T\circ T_1\in \mascI _p(\mascB _{0,2},\mascB _2)$,
and
\begin{equation}\label{Eq:SchattComp}
\nm {T_2\circ T\circ T_1}{\mascI _p(\mascB _{0,2},\mascB _2)}
\lesssim \nm {T_1}{\maclB (\mascB _{0,2},\mascB _{0,1})}\nm {T_2}{\maclB (\mascB _1,\mascB _2)}
\nm T{ \mascI _p(\mascB _{0,1},\mascB _1)}\text ;
\end{equation}

\vrum

\item if $T\in \mascN _r(\mascB _{0,1},\mascB _1)$, then
$T_2\circ T\circ T_1\in \mascN _r(\mascB _{0,2},\mascB _2)$,
and
\begin{equation}\label{Eq:NuclComp}
\nm {T_2\circ T\circ T_1}{\mascN _r(\mascB _{0,2},\mascB _2)}
\le \nm {T_1}{\maclB (\mascB _{0,2},\mascB _{0,1})}\nm {T_2}{\maclB (\mascB _1,\mascB _2)}
\nm T{ \mascN _r(\mascB _{0,1},\mascB _1)}\text ;
\end{equation}

\vrum

\item if in addition $\mascB _{0,1}$ and $\mascB _{1}$ are Hilbert spaces,
then $\mascN _r(\mascB _{0,1},\mascB _{1}) =
\mascI _r(\mascB _{0,1},\mascB _{1})$, with equality in quasi-norms.
\end{enumerate}
\end{prop}

\par

Proposition \ref{Prop:NuclQBanach} is well-known in the literature (cf. \cite{BS, Si, Toft21} and
the references therein).

\par

\subsection{Modulation spaces}

\par

Next we discuss basic properties for modulation spaces, and start by
recalling the conditions for the involved weight functions. A function $\omega$
on $\rr d$ is called a \emph{weight} (on $\rr d$), if $\omega >0$ and
$\omega ,\omega ^{-1}\in L^\infty _{loc}(\rr d)$. Let $\omega$ and $v$
be weights on $\rr d$. Then $\omega$ is called \emph{moderate} or
\emph{$v$-moderate} if
\begin{equation}\label{moderate}
\omega (x+y) \lesssim \omega (x)v(y),\quad x,y\in \rr d.
\end{equation}
The weight $v$ is called \emph{submultiplicative},
if $v$ is even and \eqref{moderate} holds when
$\omega =v$. We note that if \eqref{moderate} holds, then
$$
v(-x)^{-1}\lesssim \omega (x) \lesssim v(x).
$$
Furthermore, for such $\omega$ it follows that \eqref{moderate} is true when
$$
v(x) =Ce^{r|x|},
$$
for some positive constants $r$ and $C$ (cf. e.{\,}g. \cite{Gc2.5}).

\par

The set of all moderate functions on $\rr d$
is denoted by $\mascP _E(\rr d)$.

\par

Let $\omega \in \mascP _E (\rr {2d})$ and $p,q\in (0,\infty ]$
be fixed. Then the mixed
Lebesgue space $L^{p,q}_{(\omega )}(\rr {2d})$ consists of
all measurable functions $F$ on $\rr {2d}$ such that
$\nm F{L^{p,q}_{(\omega )}}<\infty$. Here
\begin{equation}\label{Lpq1norm}
\nm F{L^{p,q}_{(\omega )}} \equiv \nm {F_{p,\omega}}{L^q},
\quad \text{where} \quad
F_{p,\omega}(\xi ) \equiv \nm {F(\cdo ,\xi )\omega (\cdo ,\xi )}{L^p}.
\end{equation}
We note that these quasi-norms might attain $+\infty$.

\par

Let $\phi \in \maclS _{1/2} (\rr d)\setminus 0$ be fixed.
The \emph{modulation space} $M^{p,q}_{(\omega )}(\rr d)$
is the space which consist of
all $f\in \maclS _{1/2} '(\rr d)$ such that $\nm f{M^{p,q}_{(\omega
)}}<\infty$, where
\begin{equation}\label{modnorm}
\nm f{M^{p,q}_{(\omega )}}\equiv \nm {V_\phi f}{L^{p,q}_{(\omega
)}}.
\end{equation}

\par

For convenience we set $M^p _{(\omega )}= M^{p,p}_{(\omega
)}$. Furthermore we set $M^{p,q}=M^{p,q}_{(\omega )}$ when
$\omega \equiv 1$.

\par

The proof of the following proposition is omitted, since the results
can be found in \cite {Fe4,FG1,FG3,GaSa,Gc2}. Here, if
$p\in[1,\infty]$, then $p'\in[1,\infty]$ is the conjugate exponent of $p$.
That is, $p$ and $p'$ should satisfy $\frac 1p+\frac 1{p'}=1$.

\par

\begin{prop}\label{p1.4}
Let $p,q,p_j,q_j\in (0,\infty ]$ for $j=1,2$, $r\le \min (p,q,1)$,
and $\omega ,\omega _1,\omega _2,v\in \mascP _E (\rr {2d})$
be such that $v$ is submultiplicative, $\omega$ is $v$-moderate and
$\omega _2\lesssim \omega _1$. Then the following is true:
\begin{enumerate}
\item[(1)] $f\in M^{p,q}_{(\omega )}(\rr d)$ if and only if
\eqref{modnorm} holds for any $\phi \in M^r_{(v)}(\rr d)\setminus
0$. Moreover, $M^{p,q}_{(\omega )}$ is a quasi-Banach space under
the quasi-norm
in \eqref{modnorm} and different choices of $\phi$ give rise to
equivalent quasi-norms. Furthermore, if $p,q\ge 1$, then $M^{p,q}_{(\omega )}$
is a Banach space;

\vrum

\item[(2)] if  $p_1\le p_2$ and $q_1\le q_2$  then
$$
\Sigma _1 (\rr d)\hookrightarrow M^{p_1,q_1}_{(\omega _1)}(\rr
d)\hookrightarrow M^{p_2,q_2}_{(\omega _2)}(\rr d)\hookrightarrow
\Sigma _1 '(\rr d)\text ;
$$

\vrum

\item[(3)] if in addition $p,q\ge 1$, then the $L^2$ product
$( \cdo ,\cdo )_{L^2}$ on $\maclS _{1/2}(\rr d)$
extends uniquely  to a continuous map from $M^{p,q}_{(\omega )}(\rr
n)\times M^{p'\! ,q'}_{(1/\omega )}(\rr d)$ to $\mathbf C$. On the
other hand, if $\nmm a = \sup \abp {(a,b)}$, where the supremum is
taken over all $b\in \maclS _{1/2} (\rr d)$ such that
$\nm b{M^{p',q'}_{(1/\omega )}}\le 1$, then $\nmm {\cdot}$ and $\nm
\cdot {M^{p,q}_{(\omega )}}$ are equivalent norms;

\vrum

\item[(4)] if $p,q<\infty$, then $\maclS _{1/2} (\rr d)$ is dense in
$M^{p,q}_{(\omega )}(\rr d)$. If in addition $p,q\ge 1$, then
the dual space of $M^{p,q}_{(\omega )}(\rr d)$ can be identified
with $M^{p'\! ,q'}_{(1/\omega )}(\rr d)$, through the $L^2$-form
$(\cdo  ,\cdo )_{L^2}$. Moreover, $\maclS _{1/2} (\rr d)$ is weakly dense
in $M^{p' ,q'}_{(\omega )}(\rr d)$ with respect to the $L^2$-form.
\end{enumerate}
\end{prop}

\par

\begin{rem}\label{remGSmodident}
By Theorem 3.9 in \cite{Toft8} it follows that Gelfand-Shilov spaces
and their distribution spaces can be obtained by suitable unions and
intersections of modulation spaces. In particular we have
$$
\bigcap _{\omega \in \mascP _E}M^{p,q}_{(\omega )}(\rr d) = \Sigma _1(\rr d),
\quad
\bigcup _{\omega \in \mascP _E}M^{p,q}_{(\omega )}(\rr d) = \Sigma _1'(\rr d).
$$
\end{rem}

\par

\section{Born-Jordan operators with distribution symbols}\label{sec2}

\par

In this section we deduce various kinds of mapping properties of $\op _{\BJ}(a)$
when $a$ belongs to suitable test-function or distribution spaces. In particular we
show that $\op _{\BJ}(a)$ makes sense as a continuous operator from
$\mascS (\rr d)$ to $\mascS '(\rr d)$ when $a\in \mascS '(\rr {2d})$.

\par

We begin with the following analogy of Proposition \ref{Prop:WignerDistExt}
in Born-Jordan situation.

\par

\begin{prop}\label{Prop:WignerBJDistExt}
Let $s,\sigma >0$ be such that $s+\sigma \ge 1$ and let $T(f,\phi )
\equiv W_{\BJ} (f,\phi )$ when $f,\phi \in \mascS (\rr d)$. Then the following
is true:
\begin{enumerate}
\item the map $(f,\phi)\mapsto T(f,\phi )$ is continuous from
$\mascS (\rr d)\times \mascS (\rr d)$ to
$\mascS (\rr {2d})$ and restricts to a continuous map from
$\maclS _s^\sigma (\rr d)\times \maclS _s^\sigma
(\rr d)$ to $\maclS _{s,\sigma}^{\sigma ,s}(\rr {2d})$;


\item the map $(f,\phi)\mapsto T(f,\phi )$ from
$\maclS _s^\sigma (\rr d)\times \maclS _s^\sigma
(\rr d)$ to $\maclS _{s,\sigma}^{\sigma ,s}(\rr {2d})$ extends
uniquely to a continuous map
from $(\maclS _s^\sigma )'(\rr d)\times
(\maclS _s^\sigma )'(\rr d)$ to
$(\maclS _{s,\sigma}^{\sigma ,s})'(\rr {2d})$ and from
$\mascS '(\rr d)\times \mascS '(\rr d)$ to
$\mascS '(\rr {2d})$.
\end{enumerate}
The same holds true for $(s,\sigma )\neq (\frac 12,\frac 12)$
when each $\maclS _s^\sigma$ and $\maclS _{s,\sigma}^{\sigma ,s}$
are replaced by $\Sigma _s^\sigma$ and
$\Sigma _{s,\sigma}^{\sigma ,s}$, respectively.
\end{prop}

\par

\begin{proof}
We use the same notations as in the proof of Proposition \ref{Prop:WignerDistExt}.
We recall that $W_{BJ} (f,g)$ is given by \eqref{Eq:WignerBJIntForm}
when $f,g\in \mascS (\rr d)$. By Proposition \ref{Prop:WignerDistExt}
and Remark \ref{Rem:WignerDistExt} it follows that the map
$(f,\phi )\mapsto W_{BJ} (f,g)$ is continuous from
$\mascS (\rr d)\times \mascS (\rr d)$
to $\mascS (\rr {2d})$, and that the same holds true if each
$\mascS$ is replaced by $\maclS _s$ or by $\Sigma _s$.

\par

In the same way, Proposition \ref{Prop:WignerDistExt}
and Remark \ref{Rem:WignerDistExt} show that the map
$$
(f,g)\mapsto \int _0^1W_t (f,g)\, dt
$$
is well-defined and continuous from $\mascS '(\rr d)\times \mascS '(\rr d)$
to $\mascS '(\rr {2d})$, and that the same holds true after each
$\mascS (\rr d)$, $\mascS (\rr {2d})$ and their duals are replaced by
$\maclS _s^\sigma (\rr d)$, $\maclS _{s,\sigma}^{\sigma ,s}(\rr {2d})$
and their duals, or by 
$\Sigma _s^\sigma (\rr d)$, $\Sigma _{s,\sigma}^{\sigma ,s}(\rr {2d})$
and their duals (cf. \cite{AbCaTo}). Hence, by letting
$W_{\BJ}(f,g)$ be defined by
the right-hand side of \eqref{Eq:WignerBJIntForm} for such $f$ and $g$, the asserted
continuity of the extensions of the map $T$ follows.

\par

It remains to show the asserted uniqueness of the extensions of $T$
and we only prove the uniqueness when
$f,g\in (\maclS _s^\sigma )'(\rr d)$. The cases when
$f,g\in (\Sigma _s^\sigma )'(\rr d)$ or $f,g\in \mascS '(\rr d)$
follow by similar arguments and are left for the reader.
Let $f,g\in (\maclS _s^\sigma )'(\rr d)$. By Proposition
\ref{Prop:WignerDistExt} and its proof, and Remark
\ref{Rem:WignerDistExt}, it follows that if
$$
\{ f_\ep \} _{\ep >0} \subseteq (\maclS _s^\sigma )'(\rr d)
\quad \text{and}\quad
\{ g_\ep \} _{\ep >0} \subseteq (\maclS _s^\sigma )'(\rr d),
$$
are such that
$$
\lim _{\ep \to 0+}f_\ep = f
\quad \text{and}\quad
\lim _{\ep \to 0+}g_\ep = g
$$
with convergence in $(\maclS _s^\sigma )'(\rr d)$, then 
$$
\lim _{\ep \to 0+}W_t (f_\ep ,g_\ep ) = W_t (f,g)
$$
in $(\maclS _s^\sigma )'(\rr {2d})$, locally uniformly with respect
to $t$. Hence
$$
\lim _{\ep \to 0+}W_{\BJ} (f_\ep ,g_\ep ) = W_{\BJ} (f,g)
$$
with convergence in $(\maclS _s^\sigma )'(\rr {2d})$.
Hence, if $\Phi \in \maclS _{s,\sigma}^{\sigma ,s}(\rr {2d})$,
then Fubini's theorem gives
\begin{multline*}
\lim _{\ep \to 0+}\scal {W_{\BJ}(f_\ep ,g_\ep)}\Phi
=
\lim _{\ep \to 0+}\Scal {\int _0^1 (\mascF _2\circ U_t )
(f_\ep \otimes \overline{g_\ep})\, dt}\Phi
\\[1ex]
\lim _{\ep \to 0+}\Scal {
f_\ep \otimes \overline{g_\ep}}{\int _0^1 (U_t ^{-1}
\circ \mascF _2 ^{-1})\Phi \, dt}
=
\Scal {
f\otimes \overline g }{\int _0^1 (U_t ^{-1}\circ \mascF _2 ^{-1})\Phi \, dt},
\end{multline*}
where the last equality follows from the fact that
$$
\Phi \mapsto \int _0^1 (U_t ^{-1}\circ \mascF _2 ^{-1})\Phi \, dt
$$
is continuous from $\maclS _{s,\sigma}^{\sigma ,s}(\rr {2d})$
to $\maclS _{s}^{\sigma}(\rr {2d})$, in view of Proposition
\ref{Prop:WignerDistExt} and its proof.
The uniqueness assertions now follow from the facts that
we may choose $f_\ep$ and $g_\ep$ in $\maclS _s^\sigma (\rr d)$.
This gives the result.
\end{proof}

\par

We have now the following. Here $\mathscr L(V_1,V_2)$ is the set of all
linear and continuous mappings from the topological vector space $V_1$
into the topological vector space $V_2$.

\par

\begin{thm}\label{Thm:BJopsDistrMap}
Let $s,\sigma >0$ be such that $s+\sigma \ge 1$.
Then the following is true:
\begin{enumerate}
\item if $a\in \maclS _{s,\sigma}^{\sigma ,s}(\rr {2d})$, then
$\op _{\BJ}(a)$ from $\maclS _s^\sigma (\rr d)$ to
$(\maclS _s^\sigma )'(\rr d)$ is uniquely extendable
to a continuous map from $(\maclS _s^\sigma )'(\rr d)$ to
$\maclS _s^\sigma (\rr d)$;

\vrum

\item the map $a\mapsto \op _{\BJ}(a)$ from
$\maclS _{s,\sigma}^{\sigma ,s}(\rr {2d})$
to $\mathscr L(\maclS _s^\sigma (\rr d),
(\maclS _s^\sigma )'(\rr d))$ is uniquely extendable
from $(\maclS _{s,\sigma}^{\sigma ,s})'(\rr {2d})$ to
$\mathscr L(\maclS _s^\sigma (\rr d), (\maclS _s^\sigma )'(\rr d))$.
\end{enumerate}
The same holds true if each $\maclS _s^\sigma$,
$\maclS _{s,\sigma}^{\sigma ,s}$ and their duals are replaced by
$\Sigma _s^\sigma$ and $\Sigma _{s,\sigma}^{\sigma ,s}$
and their duals, or by $\mascS$ and its dual.
\end{thm}

\par

\begin{proof}
We only prove the assertion for symbols in
$\maclS _{s,\sigma}^{\sigma ,s}(\rr {2d})$ and
$(\maclS _{s,\sigma}^{\sigma ,s})'(\rr {2d})$.
The other cases follow by similar arguments and are left
for the reader.

\par

By Proposition \ref{Prop:WignerBJDistExt} it follows that the map
$$
(a,f,g)\mapsto (2\pi )^{-d/2}(a,W_{\BJ} (g,f))_{L^2(\rr {2d})}
$$
from
$\maclS _{s,\sigma}^{\sigma ,s}(\rr {2d})\times \maclS _s^\sigma
(\rr d)\times \maclS _s^\sigma (\rr d)$ to $\mathbf C$
extends uniquely to continuous mappings from
$\maclS _{s,\sigma}^{\sigma ,s}(\rr {2d})\times
(\maclS _s^\sigma )'(\rr d)\times (\maclS _s^\sigma )'(\rr d)$
to $\mathbf C$,
and from
$(\maclS _{s,\sigma}^{\sigma ,s})'(\rr {2d})
\times \maclS _s^\sigma (\rr d)\times \maclS _s^\sigma (\rr d)$
to $\mathbf C$.
Hence, by letting $\op _{BJ}(a)f$ be defined by \eqref{Eq:WignerPseudoRel},
the asserted continuity follows. The uniqueness of these extensions follows by
similar arguments to those of Proposition \ref{Prop:WignerBJDistExt}. The details
are left to the reader.
\end{proof}

\par

We may now complete the previous result with the following.

\par

\begin{thm}\label{BJGSSymbols}
Let $s,\sigma >0$ be such that $s+\sigma \ge 1$, and let
$t\in \mathbf R$. Then the following is true:
\begin{enumerate}
\item if $a\in (\maclS _{s,\sigma}^{\sigma ,s})'(\rr {2d})$, then $\op _{\BJ}(a)=\op _t(b)$, for some $b\in (\maclS _{s,\sigma}^{\sigma ,s})'(\rr {2d})$;

\vrum

\item if $a\in \maclS _{s,\sigma}^{\sigma ,s}(\rr {2d})$, then
$\op _{\BJ}(a)=\op _t(b)$, for some $b\in \maclS _{s,\sigma}^{\sigma ,s}
(\rr {2d})$.
\end{enumerate}

\par

The same holds true with $\mascS$ or $\Sigma _{s,\sigma}^{\sigma ,s}$
in place of $\maclS _{s,\sigma}^{\sigma ,s}$ at each occurrence.
\end{thm}

\par

\begin{proof}
We only prove the assertion for $\maclS _{s,\sigma}^{\sigma ,s}$
and $(\maclS _{s,\sigma}^{\sigma ,s})'$.
The other cases follow by similar arguments and are left for the reader.

\par

By \cite[Theorem 3.6]{AbCaTo} it suffices to prove the result in
the Weyl case $t=\frac 12$.
Let $\Omega _1$ and $\Omega _2$ be bounded sets in
$\maclS _{s,\sigma}^{\sigma ,s}(\rr {2d})$ and
in $(\maclS _{s,\sigma}^{\sigma ,s})'(\rr {2d})$, respectively,
and let $I\subseteq \mathbf R$ be a bounded
interval. Then the map $(t,a)\mapsto e^{it\scal {D_\xi}{d_x}}a$ is uniformly
continuous from $I\times \Omega _1$ to
$\maclS _{s,\sigma}^{\sigma ,s}(\rr {2d})$ and from
$I\times \Omega _2$ to
$(\maclS _{s,\sigma}^{\sigma ,s})'(\rr {2d})$, in view of
\cite[Theorem 3.6]{AbCaTo} and its proof. Hence
$$
a_{\BJ} \equiv \int _0^1e^{i(\frac 12-t)\scal {D_\xi}{D_x}}a\, dt
$$
belongs to $\maclS _{s,\sigma}^{\sigma ,s}(\rr {2d})$ respective
$(\maclS _{s,\sigma}^{\sigma ,s})'(\rr {2d})$ when $a$ does.
The result now follows from \eqref{Eq:PseudoBJ1},
\eqref{Eq:CalculTransfer} and the uniqueness assertions
in Theorem \ref{Thm:BJopsDistrMap}.
\end{proof}

\par

\section{Born-Jordan operators with modulation space symbols}\label{sec3}

\par

In this section we deduce that any Born-Jordan operator with symbol in
the modulation space $M^{p,q}_{(\omega )}(\rr {2d})$ is a pseudo-differential operator
with symbol in $M^{p,q}_{(\omega )}(\rr {2d})$, when $\omega (x,\xi ,\eta ,y)
=\omega _0(\eta ,y)$. We also deduce continuity,
Schatten-von Neumann and nuclearity properties for such operators. 

\par

We begin with the following.

\par

\begin{thm}\label{BJModSymbols}
Let $p,q\in (0,\infty ]$, $t \in \mathbf R$, $\omega \in
\mascP _E(\rr {4d})$
be such that $\omega (x,\xi ,\eta ,y)=\omega _0(\eta ,y)$ for some
$\omega _0\in \mascP _E(\rr {2d})$, and let
$a\in M^{p,q}_{(\omega )}(\rr {2d})$. Then $\opBJ (a) =\opp {t } (b)$
for some $b\in M^{p,q}_{(\omega )}(\rr {2d})$.
\end{thm}

\par

For the proof we recall that if $\mu$ and $\nu$ are positive
measures on measurable spaces, $p,q\in (0,\infty ]$ satisfy $q\le p$
and $f$ is $\mu\times \nu$ measurable, then
Minkowski's inequality asserts that
$$
\nm G{L^p(d\nu )} \le \nm H{L^q(d\mu )}
$$
when
$$
G(y) \equiv \nm {f(\cdo ,y)}{L^q(d\mu )} 
\quad \text{and}\quad
H(x) \equiv \nm {f(x,\cdo )}{L^p(d\nu )}.
$$

\par

In order to treat the case $p\in (0,1]$ in suitable ways, we need the following
lemma.

\par

\begin{lemma}\label{Lemma:BJModSymbols}
Let $y\in \rr d$, $\Lambda \subseteq \rr d$ be a lattice and
$p\in (0,\infty)$. Then
\begin{align}
\int _{\rr d}
\left (
\int _0^1 e^{-r|x-ty|}\, dt
\right )^p
\, dx
&\le
\left (
\frac {10\sqrt d}{pr}
\right )^d
e^{-\frac {pr}{8\sqrt d}\cdot |y|}.
\label{Eq:IntExpEst}
\intertext{and}
\sum _{\mabfj \in \Lambda}
\left (
\int _0^1 e^{-r|\mabfj -ty|}\, dt
\right )^p
\, dx
&\lesssim
e^{-\frac {pr}{8\sqrt d}\cdot |y|}.
\label{Eq:IntExpEst2}
\end{align}
\end{lemma}

\par

\begin{proof}
It is clear that both sides of \eqref{Eq:IntExpEst} are even functions
with respect to each $y_j$ in $y=(y_1,\dots ,y_d)$. Hence we may assume
that $y_j\ge 0$ for every $j$, when proving \eqref{Eq:IntExpEst}. We prove
only \eqref{Eq:IntExpEst}. The estimate \eqref{Eq:IntExpEst2}
follows by similar arguments and is left for the reader.

\par

Since $u\mapsto e^{-ru}$ is a convex function,
H{\"o}lder's and Jensen's inequalities give
\begin{multline}\label{Eq:ExpIntHoldJensen}
\int _{\rr d}
\left (
\int _0^1 e^{-r|x-ty|}\, dt
\right )^p
\, dx
\le
\int _{\rr d}
\left (
\int _0^1 \prod _{j=1}^d e^{-\frac r{\sqrt d}\cdot |x_j-ty_j|}\, dt
\right )^p
\, dx
\\[1ex]
\le
\int _{\rr d}
\prod _{j=1}^d 
\left (
\int _0^1 e^{-r{\sqrt d}\cdot |x_j-ty_j|}\, dt
\right )^{\frac pd}
\, dx
\le 
\prod _{j=1}^d 
\int _{\mathbf R}h(x_j,y_j)\, dx_j,
\end{multline}
where
\begin{equation}\label{Eq:hFunctDef}
h(u_1,u_2) = \exp
\left (
-\frac {pr}{\sqrt d}\int _0^1 |u_1-tu_2|\, dt
\right )
,\quad u_1\in \mathbf R,\ u_2\in \overline{\mathbf R_+}\, .
\end{equation}

\par

We need to evaluate the integral in \eqref{Eq:hFunctDef}. By straight-forward
computations we get
$$
\int _0^1 |u_1-tu_2|\, dt
=
\begin{cases}
\frac {2u_1-u_2}2, & u_1\ge u_2,
\\[1ex]
\frac {2u_1^2-2u_1u_2+u_2^2}{2u_2}, & 0\le u_1\le u_2,
\\[1ex]
\frac {u_2-2u_1}2, & u_1<0.
\end{cases}
$$
This gives
$$
\int _{\mathbf R}h(u_1,u_2)\, du_1
=
I_1(u_2)+I_2(u_2)+I_3(u_2),
$$
where
\begin{align*}
I_1(u_2) &= \int _{u_2}^\infty e^{-\frac {pr}{\sqrt d}\cdot \frac {2u_1-u_2}2}\, du_1
=
\frac {\sqrt d}{pr}e^{-\frac {pr}{2\sqrt d}u_2},
\\[1ex]
I_2(u_2) &= \int _0^{u_2}
e^{-\frac {pr}{\sqrt d}\cdot \frac {2u_1^2-2u_1u_2+u_2^2}{2u_2}}\, du_1
\le
\int _0^{u_2}
e^{-\frac {pr}{\sqrt d}\cdot \frac {2u_1^2-2u_1u_2+u_2^2}{2u_2}}\, du_1
\\[1ex]
&\le
\int _0^{u_2} e^{-\frac {pr}{4\sqrt d}\cdot u_2}\, du_1 =
u_2e^{-\frac {pr}{4\sqrt d}\cdot u_2}
\le
\frac {8\sqrt d}{pr}e^{-\frac {pr}{8\sqrt d}\cdot u_2}
\intertext{and}
I_1(u_2) &= \int _{-\infty}^0
e^{-\frac {pr}{\sqrt d}\cdot \frac {u_2-2u_1}2}\, du_1
=
\frac {\sqrt d}{pr}e^{-\frac {pr}{2\sqrt d}u_2}.
\end{align*}

\par

By combining these estimates we get
$$
\int _{\mathbf R}h(u_1,u_2)\, du_1
\le
\frac {10\sqrt d}{pr}e^{-\frac {pr}{8\sqrt d}\cdot u_2}.
$$
Hence, \eqref{Eq:ExpIntHoldJensen} and the fact that $y_j\ge 0$
for every $j$ give
\begin{multline*}
\int _{\rr d}
\left (
\int _0^1 e^{-r|x-ty|}\, dt
\right )^p
\, dx
\le
\left (
\frac {10\sqrt d}{pr}
\right )^d
e^{-\frac {pr}{8\sqrt d}(y_1+\cdots +y_d)}
\\[1ex]
\le
\left (
\frac {10\sqrt d}{pr}
\right )^d
e^{-\frac {pr}{8\sqrt d}|y|},
\end{multline*}
and the result follows.
\end{proof}

\par

\begin{proof}[Proof of Theorem \ref{BJModSymbols}]
By the assumptions we have
$$
a_1\in M^{p,q}_{(\omega )}(\rr {2d})
\quad \Leftrightarrow \quad
a_2\in M^{p,q}_{(\omega )}(\rr {2d})
$$
when $\op _{t_1}(a_1)=\op _{t_2}(a_2)$, in view of \cite[Proposition 1.7]{Toft9}.
Hence it suffices to prove the result in the case $t =0$.
We also assume that $p,q<\infty$. The cases when
$p=\infty$ or $q=\infty$ follow by similar arguments and
are left for the reader.

\par

Let $r=\min (1,q)$, $\Lambda _\ep =\ep \zz {2d}$ and
$\Lambda _\ep ^2=\Lambda _\ep\times \Lambda _\ep$
when $\ep >0$. By \cite{Toft9} there are
$v\in \mascP _E(\rr {4d})$ which is submultiplicative
such that $\omega$ is $v$-moderate,
$$
\Psi \in \Sigma _1(\rr {2d})
\quad \text{and}\quad
\Psi _{\! 0}\in M^r_{(v)}(\rr {2d})
$$
such that
\begin{align}
a(X) &= \sum _{\mabfj ,\mabfk \in \Lambda _\ep}
c(\mabfj ,\mabfk)\Psi (X-\mabfj )
e^{i\scal X{\rho (\mabfk )}},\label{Eq:GaborExp}
\intertext{and}
\nm c{\ell ^{p,q}_{(\omega)}(\Lambda _\ep ^2)}
&\asymp
\nm a{M^{p,q}_{(\omega )}},
\notag
\intertext{where}
c(\mabfj ,\mabfk) &= (V_{\Psi _{\! 0}}a)(\mabfj ,\rho (\mabfk ))
\label{Eq:GaborCoeff}
\end{align}
when $a\in M^{p,q}_{(\omega )}(\rr {2d})$, provided $\ep >0$ is
chosen small enough.
Here $\rho$ is the reflexion operator on $\rr {2d}$ given by
$\rho (x,\xi )=(\xi ,x)$ when $x,\xi \in \rr d$.

\par

For $a\in M^{p,q}_{(\omega )}(\rr {2d})$ and
$\Phi \in \Sigma _1(\rr {2d})\setminus 0$ fixed we now get
\begin{align}
V_\Phi (e^{it \scal {D_\xi}{D_x}}a) &= \sum
_{\mabfj ,\mabfk \in \Lambda _\ep}
c(\mabfj ,\mabfk )H_{\mabfj ,\mabfk},
\label{Eq:STFTExp1}
\intertext{where}
H_{\mabfj ,\mabfk} &= V_\Phi (e^{it \scal {D_\xi}{D_x}}
(\Psi (\cdo -\mabfj )e^{i\scal \cdo {\rho (\mabfk)}}).
\notag
\end{align}
We need to simplify $H_{\mabfj ,\mabfk}$. By the definitions and
straight-forward computations, using Fourier's inversion formula
we get
$$
H_{\mabfj ,\mabfk}(X,\rho (Y))=e^{i(\scal X{\rho (\mabfk -Y)} +it\scal k\kappa}
 (V_\Phi \Psi _t )(X-\mabfj +t \mabfk ,\rho (Y-\mabfk )),
$$
where
$$
\Psi _t = e^{it \scal {D_\xi}{D_x}}\Psi
\quad \text{and}\quad
\mabfk =(k,\kappa )
$$
(see (1.13) in \cite{Toft6}, and its proof).

\par

Since $\Psi \in \Sigma _1(\rr {2d})$, it follows from \cite{CaTo,Tr} that
$\sets {\Psi _t}{t\in [0,1]}$ is a bounded set in $\Sigma _1(\rr {2d})$.
Hence, by \cite{ChuChuKim}, for every $r>0$ there is a constant $C_r$
which is independent of $t \in [0,1]$ such that
$$
|H_{\mabfj ,\mabfk}(X,\rho (Y))|
=
|(V_\Phi \Psi _t )(X-\mabfj +t \mabfk ,\rho (Y-\mabfk ))|
\le C_re^{-r(|X-\mabfj +t \mabfk |+|Y-\mabfk |)}.
$$

\par

By using the latter estimate in \eqref{Eq:STFTExp1} we get
with $r>0$ large enough that
\begin{multline}\label{Eq:BasicGaborExpEst}
|V_\Phi b(X,\rho (Y))|\omega _0(\rho (Y))
\\[1ex]
\lesssim
\sum _{\mabfj ,\mabfk \in \Lambda _\ep}|c(\mabfj ,\mabfk )|
\int _0^1e^{-r(|X-\mabfj +t \mabfk |}\, dt e^{-2r|Y-\mabfk |}
\omega _0(\rho (Y))
\\[1ex]
\lesssim
\sum _{\mabfj ,\mabfk \in \Lambda _\ep}|c_0(\mabfj ,\mabfk )|
\int _0^1e^{-r(|X-\mabfj +t \mabfk |}\, dt e^{-r|Y-\mabfk |},
\end{multline}
where $c_0(\mabfj ,\mabfk ) =c(\mabfj ,\mabfk )\omega _0(\rho (\mabfk))$.
We shall now consider the two cases $p\ge 1$ and $p<1$
separately.

\par

First suppose that $p\ge 1$. By Minkowski's inequality we get
$$
\nm {V_\Phi b(\cdo ,\rho (Y))}{\ell ^p(\Lambda _\ep )}\omega _0(\rho (Y))
\lesssim
\sum _{\mabfk \in \Lambda _\ep} h(\mabfk )e^{-r|Y-\mabfk |}
$$
where
$$
h(\mabfk )
=
\int _0^1
\left (
\sum _{\mabfm \in \Lambda _\ep}
\left (
\sum _{\mabfj \in \Lambda _\ep}
|c_0(\mabfj ,\mabfk )|e^{-r|\mabfm -\mabfj+t \mabfk |}
\right )^p
\right )^{1/p}\, dt .
$$
We have
\begin{multline*}
h(\mabfk )
\lesssim
\int _0^1
\left (
\sum _{\mabfm \in \Lambda _\ep}
\left (
\sum _{\mabfj \in \Lambda _\ep}
|c_0(\mabfj ,\mabfk )|e^{-r|\mabfm -\mabfj |}
\right )^p
\right )^{1/p}\, dt
\\[1ex]
=
\nm {c_0(\cdo ,\mabfk )*e^{-r|\cdo |}}{\ell ^p}
\lesssim
\nm {c_0(\cdo ,\mabfk )}{\ell ^p},
\end{multline*}
where the last step follows from Young's inequality. Here $*$
denotes the discrete convolution. If $h_1(\mabfk )
=\nm {c_0(\cdo ,\mabfk )}{\ell ^p}$, then
we get from these estimates that
$$
\nm {V_\Phi b(\cdo ,\rho (\mabfn ))}{\ell ^p(\Lambda _\ep )}\omega _0(\rho (\mabfn ))
\lesssim
\sum _{\mabfk \in \Lambda _\ep} h_1(\mabfk )e^{-r|\mabfn -\mabfk |}
= (h_1*e^{-r|\cdo |})(\mabfn )
$$
By applying the $\ell ^q$ norm on the last inequality we get
\begin{multline*}
\nm b{M^{p,q}_{(\omega )}}
\asymp
\nm c{\ell ^{p,q}_{(\omega )}}
\\[1ex]
\lesssim
\nm {h_1*e^{-r|\cdo |}}{\ell ^q}
\le
\nm {h_1}{\ell ^q}\nm {e^{-r|\cdo |}}{\ell ^r}
\asymp
\nm {h_1}{\ell ^q} \asymp \nm a{M^{p,q}_{(\omega )}}.
\end{multline*}
Hence we have proved
$$
\nm b{M^{p,q}_{(\omega )}}\lesssim \nm a{M^{p,q}_{(\omega )}},
$$
and the result follows in the case $p\ge 1$.

\par

Suppose instead $p<1$ and let $h_1$ be as above.
By Lemma \ref{Lemma:BJModSymbols},
applying the $\ell ^p(\Lambda _\ep)$ norm with respect to the $X$ variable
on the inequality \eqref{Eq:BasicGaborExpEst}, and using the inequality
$|a+b|^p\le |a|^p+|b|^p$ we get
\begin{multline*}
\nm {V_\Phi b(\cdo ,\rho (Y))}{\ell ^p(\Lambda _\ep )}^p\omega _0(\rho (Y))^p
\\[1ex]
\lesssim
\sum _{\mabfj ,\mabfk \in \Lambda _\ep}|c_0(\mabfj ,\mabfk )|^p
\left (
\sum _{\mabfm \in \Lambda _\ep}
\left (
\int _0^1e^{-r(|\mabfm -\mabfj +t \mabfk |}\, dt
\right ) ^p
\right )
e^{-rp|Y-\mabfk |}
\\[1ex]
\lesssim
\sum _{\mabfj ,\mabfk \in \Lambda _\ep}|c_0(\mabfj ,\mabfk )|^p
e^{-rp|Y-\mabfk |} 
=
\sum _{\mabfk \in \Lambda _\ep}|h_1(\mabfk )|^p
e^{-rp|Y-\mabfk |}.
\end{multline*}
We now apply the $\ell ^{\frac qp}(\Lambda _\ep)$ quasi-norm on
on the latter inequality and use the Young type inequality
$$
\nm {c_1*c_2}{\ell ^{\frac qp}}
\lesssim
\nm {c_1}{\ell ^{\frac qp}}\nm {c_2}{\ell ^{\min (1,\frac qp )}}
$$
to get
\begin{multline*}
\nm b{M^{p,q}_{(\omega )}}
\asymp
\nm {V_\Phi b \cdot \omega}{\ell ^{p,q}(\Lambda _\ep \times \Lambda _\ep)}
\lesssim
\nm {h_1^p * e^{-rp|\cdo |}}{\ell ^{\frac qp}(\Lambda _\ep)}
^{\frac 1p}
\\[1ex]
\le
\left (
\nm {h_1^p}{\ell ^{\frac qp}(\Lambda _\ep)}
\nm {e^{-rp|\cdo |}}{\ell ^{\min (1,\frac qp)}(\Lambda _\ep)}
\right )
^{\frac 1p}
\asymp
\nm {h_1}{\ell ^q(\Lambda _\ep)}\asymp \nm a{M^{p,q}_{(\omega)}},
\end{multline*}
which gives the result in the case $p<1$ as well.
\end{proof}

\par

We finish this section by giving some consequences of Theorem
\ref{BJModSymbols} and well-known mapping properties for
pseudo-differential operators with symbols in modulation spaces.
The involved weight functions should satisfy
\begin{equation}\label{Eq:WeightFunctCond}
\omega (x,\xi ,\eta ,y) = \omega _0(\eta ,y)
\quad \text{and}\quad
\frac {\omega _2(x,\eta -\xi )}{\omega _1(x+y,\eta +\xi )}
\lesssim \omega _0(\eta ,y).
\end{equation}

\par

The first result is a straight-forward consequence of
\cite[Theorem 3.1]{Toft19} and Theorem \ref{BJModSymbols}.
The details are left for the reader.

\par

\begin{thm}\label{Thm:BornJordOpModCont}
Let $\omega _j\in \mascP _E(\rr {2d})$, $j=0,1,2$,
$\omega \in \mascP _E(\rr {4d})$ be such that
\eqref{Eq:WeightFunctCond} holds, $p,q,p_j,q_j\in (0,\infty ]$,
$j=1,2$, be such that
$$
\frac 1{p_2}-\frac 1{p_1} = \frac 1{q_2}-\frac 1{q_1}
= \frac 1p+\min 
\left (
0,\frac 1q-1
\right ),
\quad
q\le p_2,q_2\le p,
$$
and let $a\in M^{p,q}_{(\omega )}(\rr {2d})$. Then $\op _{\BJ}(a)$
from $\Sigma _1(\rr d)$ to $\Sigma _1'(\rr d)$ is uniquely
extendable to a continuous map from $M^{p_1,q_1}_{(\omega _1)}(\rr d)$
to $M^{p_2,q_2}_{(\omega _2)}(\rr d)$.
\end{thm}

\par

The next result follows from \cite[Theorem 3.1]{Toft21} and
Theorem \ref{Thm:BJopsDistrMap}. We leave the details for the reader.

\par

\begin{thm}\label{Thm:BornJordOpSchattenBanachSp}
Let $\omega _j\in \mascP _E(\rr {2d})$, $j=0,1,2$,
$\omega \in \mascP _E(\rr {4d})$ be such that
\eqref{Eq:WeightFunctCond} holds, $p,q,r\in (0,\infty ]$
be such that
$$
\frac 1r-1\ge \max
\left (
\frac 1p-1,0
\right )
+
\left (
\frac 1q-1,0
\right )
+\frac 1q,
$$
and let $a\in M^r_{(\omega )}(\rr {2d})$. Then
$$
\op _{\BJ}(a) \in \mascI _q(M^\infty _{(\omega _1)}(\rr d),M^p_{(\omega _2)}(\rr d)).
$$
\end{thm}

\par

The following result concerns nuclearity for pseudo-differential
operators and is a consequence of \cite[Theorem 4.2]{Toft21} and
Theorem \ref{Thm:BJopsDistrMap}. The details are left
for the reader.

\par

\begin{thm}\label{Thm:BornJordOpNuclearBanachSp}
Let $\omega _j\in \mascP _E(\rr {2d})$, $j=0,1,2$,
$\omega \in \mascP _E(\rr {4d})$ be such that
\eqref{Eq:WeightFunctCond} holds, $p\in (0,1]$
and let $a\in M^p_{(\omega )}(\rr {2d})$. Then
$$
\op _{\BJ}(a) \in \mascN _p(M^\infty _{(\omega _1)}(\rr d),
M^p_{(\omega _2)}(\rr d)).
$$
\end{thm}

\par

The last result in this record concerns Schatten-von Neumann
properties for Born-Jordan operators when acting on modulation
spaces of Hilbert type. For $p\in (0,1]$, the result follows from
Theorem \ref{Thm:BornJordOpNuclearBanachSp} and the fact that
\begin{multline*}
\mascN _p(M^\infty _{(\omega _1)}(\rr d),M^p_{(\omega _2)}(\rr d))
\subseteq
\mascN _p(M^2 _{(\omega _1)}(\rr d),M^2_{(\omega _2)}(\rr d))
\\[1ex]
=
\mascI _p(M^2 _{(\omega _1)}(\rr d),M^2_{(\omega _2)}(\rr d)),
\end{multline*}
because
$$
M^2 _{(\omega _1)}(\rr d)\subseteq M^\infty _{(\omega _1)}(\rr d)
\quad \text{and}\quad
M^p_{(\omega _2)}(\rr d)\subseteq M^2_{(\omega _2)}(\rr d),\quad
p\le 2,
$$
since modulation spaces increase with their Lebesgue exponents.
For $p\in [1,2]$, the result follows from the extension
\cite[Theorem A.3]{Toft} of \cite[Theorem 4.13]{Toft6}
and Theorem \ref{Thm:BJopsDistrMap}.

\par

\begin{thm}\label{Thm:BornJordOpHilbSchatten}
Let $\omega _j\in \mascP _E(\rr {2d})$, $j=0,1,2$,
$\omega \in \mascP _E(\rr {4d})$ be such that
\eqref{Eq:WeightFunctCond} holds, $p\in (0,\infty ]$, let $q=p$ when
$p\le 2$ and $q=p'$ when $p>2$
and let $a\in M^{p,q}_{(\omega )}(\rr {2d})$. Then
$$
\op _{\BJ}(a) \in \mascI _p(M^2 _{(\omega _1)}(\rr d),M^2_{(\omega _2)}(\rr d)).
$$
\end{thm}

\par

\section{Born-Jordan operators with H{\"o}rmander symbols}\label{sec4}

\par

In this section we show that any Born-Jordan operator with symbol in a suitable
H{\"o}rmander class is a pseudo-differential operator with symbol in the same
class. Furthermore, we deduce Schatten-von Neumann properties and lower
bound estimates for such operators. Especially we deduce
Feffermann-Phong's inequality for Born-Jordan operators.

\par

First we recall the definition of the involved symbol classes. Let $g$
be a Riemannian metric on the phase space $\rr {2d}\simeq T^*\rr d$, let
$m>0$ be a function in $L^\infty _{loc}(\rr {2d})$, and let $N\ge 0$
be an integer. For any $a\in C^N(\rr {2d})$ and $X=(x,\xi )\in \rr {2d}$, let $|a|_0^g(X) =
|a(X)|$, and
$$
|a|_k^g(X)\equiv \sup |a^{(k)}(X; Y_1,\dots ,Y_k)|, \quad k\ge 1.
$$
Here the supremum is taken over all $Y_1,\dots ,Y_k\in \rr {2d}$ such that
$g_X(Y_j)\le 1$ for every $j=1,\dots ,k$. We also set
$$
\nm a{m,N}^g \equiv \sum _{k=0}^N\sup _{X\in W} (|a|_k^g(X)/m(X)).
$$
Then $S(m,g)$ consists of all $a\in C^\infty (\rr {2d})$
such that $\nm a{m,N}^g$ is finite for every $N\ge 0$. We equip
this space by the topology, induced by the semi-norms
$\nm \cdo {m,N}^g$, $N\ge 0$.

\par

The dual metric $g^\sigma$ with respect to the (standard) symplectic form
$\sigma$ is defined by
$$
g^\sigma _X(Z) \equiv \sup |\sigma (Y,Z)|^2,\quad X,Z\in \rr {2d},
$$
where the supremum is taken over all $Y\in \rr {2d}$ such that $g_X(Y)\le 1$.
Furthermore, the Planck's function $h_g(X)$ is defined by
$$
h_g(X) \equiv \sup g_X(Y)^{1/2},\quad X\in \rr {2d}.
$$
where the supremum is taken over all $Y\in \rr {2d}$ such that
$g^\sigma _X(Y)\le 1$.

\par

As in \cite{BuTo,Toft4} we need some restrictions on $m$ and $g$. More precisely,
the metric $g$ on $\rr {2d}$ is called \emph{slowly varying} if there are constants
$c,C>0$ such that
$$
C^{-1}g_X\le g_Y\le Cg_X,\quad \text{when}\quad g_X(Y)<c,
$$
and $m$ is called $g$-continuous if there are constants $c,C>0$ such that
$$
C^{-1}m(X)\le m(Y)\le Cm(X),\quad \text{when}\quad g_X(Y)<c,
$$
The metric $g$ is called \emph{$\sigma$-temperate} if there are constants $C,N>0$
such that
$$
g_X(Z)\le Cg_Y(Z)(1+g_X(X-Y))^N,\quad \text{for all}\quad X,Y,Z\in \rr {2d},
$$
and $m$ is called \emph{$(\sigma ,g)$-temperate} if there are constants $C,N>0$
such that
$$
m(X)\le Cm(Y)(1+g_X(X-Y))^N,\quad \text{for all}\quad X,Y\in \rr {2d}.
$$

\par

The metric $g$ on $W$ is called \emph{strongly feasible} if $g$ is slowly varying,
$\sigma$-temperate and $h_g\le 1$. The weight $m$ is called
\emph{$g$-feasible} if $m$ is $g$-continuous and $(\sigma ,g)$-temperate.

\par

Finally, the Riemannian metric $g$ on $\rr {2d}$ is called \emph{split} or
\emph{split metric}, if
$$
g_X(z,\zeta ) = g_X(z,-\zeta ),\quad \text{when} \quad X,(z,\zeta )\in \rr {2d}.
$$

\par

\begin{rem}\label{Rem:ClassicalSymbolClasses}
The family $S(m,g)$ may serve as a home for several classical symbol classes.
For example we have the following (see \cite{Horm} for details). Here we let 
$\eabs x=(1+|x|^2)^{\frac 12}$ when $x\in \rr d$, $X=(x,\xi )\in \rr {2d}$ and
$Y=(y,\eta )\in \rr {2d}$.
\begin{enumerate}
\item Let $r,\rho ,\delta \in \mathbf R$ satisfy $0\le \delta \le \rho \le 1$ and $\delta \neq 1$,
$$
g_X(Y) = \eabs \xi ^{2\delta} |y|^2 +\eabs \xi ^{-2\rho}|\eta |^2
\quad \text{and}\quad
m(X)=\eabs \xi ^r .
$$
Then $g$ is strongly feasible and split metric on $\rr {2d}$, $m$ is $g$-feasible
and $S(m,g)$ is equal to the H{\"o}rmander class $S^r_{\rho ,\delta}(\rr {2d})$
which consists of all $a\in C^\infty (\rr {2d})$ such that
$$
|\partial _x^\alpha \partial _\xi ^\beta a(x,\xi )|
\lesssim
\eabs \xi ^{r-\rho |\beta |+\delta |\alpha |}\text ;
$$

\vrum

\item Let $r,t,\rho ,\tau \in \mathbf R$ satisfy $0< t,\tau \le 1$,
$$
g_X(Y) = \eabs x ^{-2t} |y|^2 +\eabs \xi ^{-2\tau}|\eta |^2
\quad \text{and}\quad
m(X)=\eabs x ^r\eabs \xi ^\rho .
$$
Then $g$ is strongly feasible and split metric on $\rr {2d}$, $m$ is $g$-feasible
and $S(m,g)$ is equal to the SG class $\operatorname{SG}^{r,\rho }_{t ,\tau}(\rr {2d})$
which consists of all $a\in C^\infty (\rr {2d})$ such that
$$
|\partial _x^\alpha \partial _\xi ^\beta a(x,\xi )|
\lesssim
\eabs x ^{r-t |\alpha |}\eabs \xi ^{\rho -\tau |\beta |}\text ;
$$

\vrum

\item Let $r,t \in \mathbf R$ satisfy $0< t \le 1$,
$$
g_X(Y) = \eabs {(x,\xi )} ^{-2t} (|y|^2 +|\eta |^2 )
\quad \text{and}\quad
m(X)=\eabs {(x,\xi )} ^r.
$$
Then $g$ is strongly feasible and split metric on $\rr {2d}$, $m$ is $g$-feasible
and $S(m,g)$ is equal to the Shubin class $\operatorname{Sh}^{r}_{t}(\rr {2d})$
which consists of all $a\in C^\infty (\rr {2d})$ such that
$$
|\partial _x^\alpha \partial _\xi ^\beta a(x,\xi )|
\lesssim
\eabs {(x,\xi )} ^{r-t |\alpha +\beta |}.
$$
\end{enumerate}
\end{rem}

\par

We have now the following.

\par

\begin{thm}\label{BornJordanHormWeyl}
Let $g$ be strongly feasible and split metric on $\rr {2d}$, $m$ be $g$-feasible,
$N> 0$ be an integer, and let $a\in S(m,g)$. Then
$$
\opBJ (a) =\op ^w(b)
$$
for some $b\in S(m,g)$. Furthermore,
\begin{equation}\label{bassexp}
b-\sum _{2j<N} \frac {(-1)^j\scal {D_\xi }{D_x}^{2j}a}{4^j(2j+1)!}\in S(h_g^{N}m,g)
\end{equation}
\end{thm}

\par

For the proof we need the following proposition. We omit the proof, since
the result is essentially a restatement of
Proposition 18.5.10 in \cite{Horm}.

\par

\begin{prop}\label{Calculitransfer}
Let $g$ be strongly feasible and split metric on $\rr {2d}$, $m$ be $g$-feasible,
$I\subseteq \mathbf R$ be bounded, $N> 0$ be an integer, $a\in S(m,g)$, and let
$a_t \in \mathscr S'(\rr {2d})$ be such that
\begin{equation}\label{CalculitransferForm}
\op ^w(a)=\op _t (a_t),\qquad t \in I.
\end{equation}
Then $\{ a_t \} _{t \in I}$ is a bounded subset of $S(m,g)$, and
\begin{equation}\label{atexpansions}
\left \{
a_t  -\sum _{k<N} \left ( t -\frac 12 \right )^k
\frac {i^k\scal {D_\xi }{D_x} ^ka}{k!}
\right \}
\end{equation}
is a bounded set in $S(h_g^{N}m,g)$.
\end{prop}

\par

\begin{proof}[Proof of Theorem \ref{BornJordanHormWeyl}]
Let $a_t $ be the same as in Proposition \ref{Calculitransfer}. Then
$$
b=\int _0^1 a_t \, dt  ,
$$
and Proposition \ref{Calculitransfer} shows that $b\in S(m,g)$.
Furthermore, by \eqref{atexpansions} we get
$$
b-  \sum _{k<N} \left (\int _0^1\left ( t -\frac 12 \right ) ^k\, dt \cdot
\frac {i^k\scal {D_\xi }{D_x} ^ka}{k!} \right ) \in S(h_g^{N}m,g),
$$
which gives \eqref{bassexp}. The proof is complete.
\end{proof}

\par

Theorem \ref{BornJordanHormWeyl} shows that several continuity
properties in the Weyl calculus in Chapter XVIII in \cite{Horm} carry
over to Born-Jordan operators. For example, the following results are
immediate consequences of Propositions 18.6.2 and 18.6.3
in \cite{Horm}, and Theorem \ref{BornJordanHormWeyl}.

\par

\begin{thm}\label{BJSchwartzcont}
Let $g$ be strongly feasible and split metric on $\rr {2d}$, $m$ be $g$-feasible,
and let $a\in S(m,g)$. Then $\opBJ (a)$
is continuous on $\mascS (\rr d)$ and is uniquely extendable to
a continuous operator on $\mascS '(\rr d)$.
\end{thm}

\par

\begin{prop}\label{BJL2cont}
Let $g$ be strongly feasible and split metric on $\rr {2d}$, 
and let $a\in S(1,g)$. Then $\opBJ (a)$
is continuous on $L^2 (\rr d)$.
\end{prop}

\par

The following two theorems extend the previous result. Here
and in what follows we let $L^\infty _0(\rr d)$ be the set of all $f\in L^\infty (\rr d)$
vanishing at infinity, i.{\,}e. $f$ should satisfy
$$
\lim _{R\to \infty} \,  \underset {|x|\ge R} {\operatorname {ess\, sup}}\,  |f(x)| =0. 
$$

\par

\begin{thm}\label{HormSchattenBJ}
Let $p\in [1,\infty ]$, $t \in \mathbf R$, $g$ be strongly
feasible and split metric on $\rr {2d}$, $m$ be $g$-continuous and
$(\sigma,g)$-temperate, and let $a\in S(m,g)$. Then the following is
true:
\begin{enumerate}
\item if $h_g^{k/2}m\in L^p(\rr {2d})$ for some $k\ge 0$, and $a\in L^p(\rr {2d})$,
then $\opBJ (a)\in \mascI _p(L^2(\rr d))$;

\vspace{0.1cm}

\item if $h_g^{k/2}m\in L^\infty _0(\rr {2d})$ for some $k\ge 0$, and $a\in
L^\infty _0(\rr {2d})$, then $\opBJ (a)$ is compact on $L^2(\rr d)$.
\end{enumerate}
\end{thm}

\par

\begin{thm}\label{HormSchattenBJcor}
Let $p\in (0,\infty ]$, $t\in\mathbf R$, $g$ be strongly
feasible and split metric on $\rr {2d}$, $m$ be $g$-continuous and
$(\sigma,g)$-temperate, and let $a\in S(m,g)$. Then the following is
true:
\begin{enumerate}
\item if $m\in L^p(\rr {2d})$, then $\opBJ (a)\in \mascI _p(L^2(\rr d))$;

\vspace{0.1cm}

\item if $m\in L^\infty _0(\rr {2d})$, then $\opBJ (a)$ is compact on
$L^2(\rr d)$.
\end{enumerate}
\end{thm}

\par

For the proof of Theorem \ref{HormSchattenBJ} we need the following result
which is a slight extension of \cite[Proposition 2.10]{BuTo}. The proof is omitted,
since the result follows by the arguments in the proof of \cite[Proposition 2.10]{BuTo}.

\par

\begin{lemma}\label{intersectionlemma}
Let $p$ and $g$ be the same as in Proposition \ref{HormSchattenBJ}
and let $m$ be $g$-continuous, $(\sigma,g)$-temperate and satisfies
$h_g^{N/2}m \in L^p(\rr {2d})$ ($h_g^{N/2}m \in L^\infty _0(\rr {2d})$) for some
$N\ge 0$. Also let $t \in [0,1]$, and let $a,a_t \in \mascS '(\rr {2d})$
be related as in
\eqref{CalculitransferForm}. Then the map $a\mapsto a_t $ on
$\mathscr S'(\rr {2d})$ restricts to a continuous isomorphism on $S(m,g)\cap
L^p(\rr {2d})$ ($S(m,g)\cap L^\infty _0(\rr {2d})$), which is uniformly bounded with
respect to $t \in [0,1]$.
\end{lemma}

\par

\begin{proof}[Proof of Theorem \ref{HormSchattenBJ}]
Again let $b$ be chosen such that $\op ^w(b)= \opBJ (a)$. Then
Lemma \ref{intersectionlemma} shows that $b\in L^p(\rr {2d})\cap S(m,g)$
when $a\in L^p(\rr {2d})\cap S(m,g)$, and that $b\in L^\infty _0(\rr {2d})\cap S(m,g)$
when $b\in L^\infty _0(\rr {2d})\cap S(m,g)$. The result now follows from
Theorem 2.9 in \cite{BuTo}.
\end{proof}

\par

\begin{proof}[Proof of Theorem \ref{HormSchattenBJcor}]
The result is a special case of Theorem 4.6 in the case $p\ge 1$.
If instead $p<1$, then the result follows by combining
\cite[Theorem 4.1]{Toft19} with Theorem \ref{BornJordanHormWeyl}.
\end{proof}

\par

Finally we also have the following Feffermann-Phong's inequality inequality
for Born-Jordan operators, which in particular shows that Sharp-G{\aa}rding's
inequality is also true for such operators.

\par

\begin{thm}\label{GaardingIneq}
Let $g$ be strongly feasible and split metric on $\rr {2d}$, 
and let $0\le a\in S(h_g^{-2},g)$. Then $\opBJ (a) \ge -C$
for some constant $C\ge 0$.
\end{thm}

\par

\begin{proof}
By letting $b$ be defined by $\op ^w(b)=\op _{BJ}(a)$ choosing $N=2$ in
Theorem \ref{BornJordanHormWeyl}, it follows that $b=a+c$, where
$c\in S(1,g)$. The result now follows from the facts that $\op ^w(a)$ is lower
bounded on $L^2$ by the Feffermann-Phong's
inequality for Weyl operators (cf. \cite[Theorem 18.6.8]{Horm}),
and the fact that $\op ^w(c)$ is bounded on $L^2$ in view of
\cite[Proposition 18.6.3]{Horm}.
\end{proof}

\par

\begin{rem}
By similar arguments as in the proof of Proposition \ref{GaardingIneq}, it follows
that H{\"o}rmander's improvement of Melin's inequality given in
Theorem 6.2 in \cite{Horm0}, holds true with Born-Jordan operators in place of
Weyl operators. That is, if $g$ is strongly feasible and split metric on
$\rr {2d}$ and $a$ satisfies the same conditions as in Theorem 6.2 in \cite{Horm0},
then $\op _{\BJ}(a)$ is bounded from below. (See also the introduction.)
\end{rem}

\par

\begin{rem}
In \cite{BoCh}, Bony and Chemin introduced a broad family of Sobolev
type spaces, where each space $H(m_0,g)$ is a Hilbert space and depends on
the choice of the strongly feasible metric $g$ and the $g$-feasible weight
$m_0$. By Th{\'e}or{\`e}me 4.5 and Corollaire 6.6 in \cite{BoCh} it follows that
$H(1,g)=L^2$ and that there are $a\in S(m_0,g)$ and $b\in S(1/m_0,g)$ such that
corresponding pseudo-differential operators are inverses to each others and
lift $H(m\cdot m_0,g)$ and $H(m/m_0,g)$, respectively, to $H(m,g)$. That is,
$$
\op ^w(a)\circ \op ^w(b) = \op ^w(b)\circ \op ^w(a)
$$
equals to the identity operator on $\mascS '(\rr d)$ and
$$
\op (a)\, :\, H(m\cdot m_0,g) \to H(m,g)
\quad \text{and}\quad
\op (b)\, :\, H(m/m_0,g) \to H(m,g)
$$
are bijections for every $g$-feasible weight $m$. 

\par

From these results
it follows that Theorems \ref{HormSchattenBJ}, \ref{HormSchattenBJcor}
and \ref{GaardingIneq} can be generalized to involve $H(m,g)$ spaces.
For example,  it follows from these results and Proposition \ref{BJL2cont} that if
$g$ is a strongly feasible and split metric on $\rr {2d}$, $m$ and $m_0$ are $g$-feasible,
and $a\in S(m_0,g)$. Then $\opBJ (a)$ is continuous from $H(m_0,g)$ to $H(m_0/m,g)$.
\end{rem}

\par

\begin{rem}
Evidently, all continuity and compactness results above applies on the symbol
classes in Remark \ref{Rem:ClassicalSymbolClasses}. For example, if $a$
is symbol in a SG-class or Shubin class and belongs to $L^p(\rr {2d})$ for some
$p\in [1,\infty ]$, then Theorem \ref{HormSchattenBJcor} shows that
$\op _{\BJ}(a)\in \mascI _p(L^2(\rr d))$;
\end{rem}

\par

\section{Born-Jordan operators of infinite orders}\label{sec5}

\par

In this section we consider Born-Jordan operators with symbols
belonging to classes considered in \cite{CaTo}, or more generally,
classes considered in \cite{AbCaTo}. This means that the
symbols obey ultra-regularity conditions (i.{\,}e. certain types of
Gevrey regularity), but are allowed to grow superexponentially. In
particular, they are allowed to grow faster than polynomials. It is proved
in \cite{AbCaTo,CaTo} that corresponding pseudo-differential operators
are continuous on suitable Gelfand-Shilov spaces and their duals.
Here we use \eqref{Eq:PseudoBJ1} to carry over
these continuity properties and to Born-Jordan operators.

\par

The symbol classes which we shall consider are given in the following
definition.

\par

\begin{defn}\label{Def:ExtGSSymbClasses}
For $s_j,\sigma_j> 0$, $j=1,2$, and $h,r>0$ 
and $f\in C^\infty(\rr {d_1+d_2})$, let
\begin{equation}\label{symbols}
   \nm{f}{(h,r)}\equiv \sup \left (\frac{\abp{\partial _{x_1}
    ^{\alpha_1}\partial _{x_2}^{\alpha_2}f(x_1,x_2)}}
    {h^{\abp{\alpha_1+\alpha_2}}\alpha_1!^{\sigma _1}
    \alpha_2!^{\sigma _2}e^{r(\abp{x_1}^{\frac 1{s_1}}
    +\abp{x_2}^{\frac 1{s_2}})}} \right ), 
\end{equation}
where the supremum is taken over all $\alpha_1\in \nn {d_1},\alpha_2
\in \nn {d_2},x_1\in \rr {d_1}$ and $x_2\in \rr {d_2}$.
\begin{enumerate}
    \item $\Gamma _{s_1,s_2;0} ^{\sigma_1,\sigma_2} (\rr {d_1+d_2})$
    consists of all $f\in C^\infty(\rr {d_1+d_2})$ such that for some
    $h>0$,
    $\nm{f}{(h,r)}$ is finite for every $r>0$;
    
    \vrum
    
    \item $\Gamma _{s_1,s_2} ^{\sigma_1,\sigma_2;0} (\rr {d_1+d_2})$
    consists of all $f\in C^\infty(\rr {d_1+d_2})$ such that for some
    $r>0$, $\nm{f}{(h,r)}$ is finite for every $h>0$;
    
    \vrum
    
    \item $\Gamma _{s_1,s_2;0} ^{\sigma_1,\sigma_2;0} (\rr {d_1+d_2})$ 
    consists of all
    $f\in C^\infty(\rr {d_1+d_2})$ such that $\nm{f}{(h,r)}$ is
    finite for every $h,r>0$.
\end{enumerate}
\end{defn}

\par

The topologies of the spaces in Definition \ref{Def:ExtGSSymbClasses}
are given by canonical combinations of inductive limit and
projective limit topologies (cf. \cite{AbCaTo}).

\par

We have now the following analogies of Theorems
\ref{Thm:BJopsDistrMap} and \ref{BJGSSymbols}.

\par

\begin{thm}\label{Thm:BJopsGammaMap}
Let $s,\sigma >0$ be such that $s+\sigma \ge 1$.
Then the following is true:
\begin{enumerate}
\item if $a\in \Gamma _{s,\sigma ;0}^{\sigma ,s}(\rr {2d})$, then
$\op _{\BJ}(a)$ is continuous on $\maclS _s^\sigma (\rr d)$ and
is uniquely extendable to a continuous map on
$(\maclS _s^\sigma )'(\rr d)$;

\vrum

\item if in addition $(s,\sigma )\neq (\frac 12,\frac 12)$
and $a\in \Gamma _{s,\sigma}^{\sigma ,s;0}(\rr {2d})$, then
$\op _{\BJ}(a)$ is continuous on $\Sigma _s^\sigma (\rr d)$ and
is uniquely extendable to a continuous map on
$(\Sigma _s^\sigma )'(\rr d)$.
\end{enumerate}
\end{thm}

\par

\begin{thm}\label{BJGammaSymbols}
Let $s,\sigma >0$ be such that $s+\sigma \ge 1$, and let
$t\in \mathbf R$. If $a\in \Gamma _{s,\sigma ;0}^{\sigma ,s}(\rr {2d})$, then
$\op _{\BJ}(a)=\op _t(b)$, for some
$b\in \Gamma _{s,\sigma ;0}^{\sigma ,s}(\rr {2d})$.

\par

If in addition $(s,\sigma )\neq (\frac 12,\frac 12)$, then
the same holds true with $\Gamma _{s,\sigma}^{\sigma ,s;0}(\rr {2d})$
or $\Gamma _{s,\sigma ;0}^{\sigma ,s;0}(\rr {2d})$
in place of $\Gamma _{s,\sigma ;0}^{\sigma ,s}(\rr {2d})$ at each
occurrence.
\end{thm}

\par

Theorem \ref{BJGammaSymbols} follows from
Theorems 3.1 and 3.6 in \cite{AbCaTo} and \eqref{Eq:WeylBJ1}.
Theorem \ref{Thm:BJopsGammaMap} then follows by
combining Theorems 3.8 and 3.15 in \cite{AbCaTo} with
Theorem \ref{BJGammaSymbols}. The details
are left for the reader.

\par

\end{document}